\documentclass[11pt]{amsart}

\usepackage[T1]{fontenc}
\usepackage{amsmath, amsfonts, mathtools, amssymb, amsrefs, amsthm, tikz-cd, mathdots}
\usepackage[mathscr]{euscript}
\usepackage{genyoungtabtikz}

\usepackage[top=1in, bottom=1.25in, left=1.3in, right=1.3in]{geometry}

\usepackage[colorlinks=true,linkcolor=blue,citecolor=red]{hyperref}
\usepackage[all]{xy}
\usepackage{color}
\usepackage{tikz}

\newcommand{\nc}{\newcommand}

\theoremstyle{definition}
\newtheorem{mydef}{\textbf{Definition}}[section]
\newtheorem{myeg}[mydef]{\textbf{Example}}
\newtheorem{mythm}[mydef]{\textbf{Theorem}}
\newtheorem{rmk}[mydef]{\textbf{Remark}}

\theoremstyle{plain}
\newtheorem*{inthm}{\textbf{Theorem}}
\newtheorem{lem}[mydef]{\textbf{Lemma}}
\newtheorem{prop}[mydef]{\textbf{Proposition}}

\newcommand{\set}[2]{\left\{#1 : #2\right\}}

\nc{\mc}{\mathcal}              
\nc{\on}{\operatorname}         
\nc{\wt}{\widetilde}            

\nc{\ses}[3]{{#1}\hookrightarrow{#2}\twoheadrightarrow{#3}}

\nc{\MS}{\mathbf{Mat}_{\bullet}}

\nc{\Set}{\mathcal{S}et_{\bullet}}

\nc{\E}{\mc{E}}                 
\nc{\EE}{\mathfrak{E}}          
\nc{\MM}{\mathfrak{M}}          

\nc{\FF}{\mathbb{F}}            
\nc{\GG}{\mathbb{G}}            

\DeclareMathOperator{\id}{\on{id}}       

\nc{\I}{\mathcal{I}}
\nc{\C}{\mathcal{C}}            
\nc{\fl}{\mathbf{FL}}           
\renewcommand{\H}{{H}}       

\newcommand{\fun}{\mathbb{F}_1}

\newcommand{\A}{\on{A}}
\renewcommand{\S}{\mathcal S}

\newcommand{\mspec}{\on{MSpec}}

\newcommand{\p}{\mathfrak{p}}
\nc{\m}{\mathfrak{m}}

\newcommand{\ol}{\overline}

\nc{\h}{\mathfrak{h}}
\nc{\g}{\mathfrak{g}}
\nc{\n}{\mathfrak{n}}
\nc{\ch}{\on{CH}}

\nc{\F}{\mc{F}}
\nc{\M}{\on{M}}
\nc{\T}{\mc{T}}
\nc{\G}{\mc{G}}
\nc{\ov}{\overline}
\renewcommand{\S}{\mc{S}}

\nc{\Bun}{\mc{B}un}

\nc{\VFun}{\on{Vect}(\fun)}
\nc{\vfun}{\VFun}

\nc{\Amod}{\on{Mod}_{\on{A}}}
\nc{\sAmod}{\on{Mod}_{\on{S^{-1}A}}}

\nc{\Amoda}{\Amod^{\alpha}}

\nc{\Pone}{\mathbb{P}^1}
\nc{\Ptwo}{\mathbb{P}^2}
\nc{\Aone}{\mathbb{A}^1}
\nc{\An}{\mathbb{A}^n}
\renewcommand{\fun}{\mathbb{F}_1}
\nc{\slthat}{\widehat{\mf{sl}}_2}
\renewcommand{\a}{\mathfrak{a}}
\renewcommand{\p}{\mathfrak{p}}
\nc{\spec}{\on{MSpec}}
\nc{\Spec}{\on{Spec}}
\nc{\Msch}{\mc{M}sch}
\nc{\mt}{\widetilde{M}}

\nc{\Pn}{\fun \langle x_1, \cdots, x_n \rangle}
\nc{\Pt}{\fun \langle x_1, x_2 \rangle }

\nc{\ra}{\rightarrow}

\nc{\Qcoh}{\on{Qcoh}}
\nc{\Coh}{\on{Coh}}
\nc{\typea}{\emph{type}-$\alpha \;$}

\theoremstyle{remark}

 \numberwithin{equation}{section}

\allowdisplaybreaks[4]  

\begin{document}

 \title{Hall Lie algebras of toric monoid schemes}
  \author{Jaiung Jun and Matt Szczesny}
  \date{}
  
\makeatletter
\@namedef{subjclassname@2020}{%
	\textup{2020} Mathematics Subject Classification}
\makeatother
  

\begin{abstract}
We associate to a projective $n$-dimensional toric variety $X_{\Delta}$ a pair of co-commutative (but generally non-commutative)  Hopf algebras $H^{\alpha}_X, H^{T}_X$. These arise as Hall algebras of certain categories $\Coh^{\alpha}(X), \Coh^T(X)$ of coherent sheaves on $X_{\Delta}$ viewed as a monoid scheme - i.e. a scheme obtained by gluing together spectra of commutative monoids rather than rings. When $X_{\Delta}$ is smooth, the category $\Coh^T(X)$ has an explicit combinatorial description as sheaves whose restriction to each $\mathbb{A}^n$ corresponding to a maximal cone $\sigma \in \Delta$ is determined by an $n$-dimensional generalized skew shape. The (non-additive) categories  $\Coh^{\alpha}(X), \Coh^T(X)$ are treated via the formalism of proto-exact/proto-abelian categories developed by Dyckerhoff-Kapranov. 

The Hall algebras $H^{\alpha}_X, H^{T}_X$ are graded and connected, and so enveloping algebras $H^{\alpha}_X \simeq U(\n^{\alpha}_X)$, $H^{T}_X \simeq U(\n^{T}_X)$, where the Lie algebras $\n^{\alpha}_X, \n^{T}_X$ are spanned by the indecomposable coherent sheaves in their respective categories. 

We explicitly work out several examples, and in some cases are able to relate $\n^T_X$ to known Lie algebras. In particular, when $X = \mathbb{P}^1$, $\n^T_X$ is isomorphic to a non-standard Borel in $\mathfrak{gl}_2 [t,t^{-1}]$. When $X$ is the second infinitesimal neighborhood of the origin inside $\mathbb{A}^2$, $\n^T_X$ is isomorphic to a subalgebra of $\mathfrak{gl}_2[t]$. We also consider the case $X=\mathbb{P}^2$, where we give a basis for $\n^T_X$ by describing all indecomposable sheaves in $\Coh^T(X)$.

\end{abstract}

\maketitle

\tableofcontents

\section{Introduction}

Let $X=X_{\Delta}$ be a projective toric variety defined by a fan $\Delta$. In this paper we attach to $X$ a pair of co-commutative Hopf algebras (in fact, enveloping algebras) $H^{\alpha}_X, H^{T}_X$. These arise as Hall algebras of certain categories of coherent sheaves $\Coh^{\alpha}(X), \Coh^T(X)$ to be defined below. In order to define these, we view $X_{\Delta}$ not as an ordinary variety, but rather as a \emph{monoid scheme} - a space obtained by gluing together the spectra of commutative monoids rather than rings. One particular feature of this setting is that the category of coherent sheaves is no longer abelian, or even additive, and has an explicitly combinatorial flavor. Our constructions may be viewed within the general philosophy of algebraic geometry over $\fun$ - the "field" of one element. In the rest of this introduction we briefly recall Hall algebras of $\mathbb{F}_q$-linear finitary categories and their connections to quantum groups, discuss how Dyckerhoff-Kapranov's formalism of proto-exact categories allows one to treat Hall algebras  of non-additive categories, and how our construction may conjecturally be used to compute the classical limit of Hall algebras of toric varieties of dimension two and beyond. For an introduction to Hall algebras and their applications in representation theory, we refer the interested reader to the excellent review \cite{S}. 

\subsection{Hall algebras of finitary abelian categories and quantum groups}

An abelian (or exact) category $\mc{A}$ is called \emph{finitary} if $\on{Hom}(M,N)$ and $\on{Ext}^1(M,N)$ are finite {\bf sets} for any pair of objects $M, N \in \mc{A}$. Two examples of such are $Rep(Q, \mathbb{F}_q)$ - the category of representations of a quiver $Q$ over $\mathbb{F}_q$, and $Coh(X)$ - the category of coherent sheaves on a projective variety $X$ over $\mathbb{F}_q$. One may associate to $\mc{A}$ a pair of associative algebras $\H_{\mc{A}}, \wt{\H}_{\mc{A}}$ \cite{Ringel}. Letting $Iso(\mc{A})$ denote set of isomorphism classes of objects in $\mc{A}$, denoting by $[M]$ the isomorphism class of $M \in \mc{A}$, let $$ H_{\mc{A}} := \{ f: Iso(\mc{A}) \rightarrow \mathbb{C} \mid f \textrm{ has finite support } \} $$ 
with associative multiplication
\begin{equation} \label{hall_multiplication}
f \bullet g ([M]) = \sum_{N \subset M} f([M/N]) g([N]) 
\end{equation}
where the summation is over all sub-objects $N \subset M$. It follows that
$$ \delta_{[M]} \bullet \delta_{[N]} = \sum_{[K] \in Iso(\mc{A})} g^K_{M,N} \delta_{[K]} $$
where 
$$ g^{K}_{M,N} = | \{ L \subset K \vert L \simeq N, K/L \simeq M \} |, $$ and so 
$g^{K}_{M,N} |Aut(M)| |Aut(N)|$ counts the number of isomorphism classes of short exact sequences
$$ 0 \rightarrow N \rightarrow K \rightarrow M \rightarrow 0.$$

To connect Hall algebras to quantum groups one needs to twist the multiplication in ${H}_{\mc{A}}$ by the multiplicative Euler form
$$ \langle M, N \rangle_m := \sqrt{ \prod^{\infty}_{i=0} |\on{Ext}^i (M,N) |^{(-1)^i}  }$$ 
by introducing the new product 
$$ f \star g ([M]) = \sum_{N \subset M} \langle M/N, N \rangle_m f([M/N]) g([N]),$$ as well as extend it by $\mathbb{C}[K_0 (\mc{A})]$ (the group algebra of the Grothendieck group). 
One thereby obtains a new algebra $\wt{H}_{\mc{A}}$ defined over $\mathbb{Q}(\sqrt{q})$ when $\mathcal{A}$ is $\mathbb{F}_q$-linear. If $\mc{A}$ is hereditary, $\wt{H}_{\mc{A}}$ may be equipped with a topological bialgebra structure using the so-called Green's co-multiplication \cite{G}. 

For a quiver $Q$, and $\mc{A} = Rep(Q, \mathbb{F}_q)$, there is an embedding $U^+_{\nu}(\mathfrak{g}_Q) \hookrightarrow \wt{H}_{\mc{A}}$, where the former denotes the positive half of the quantized enveloping algebra of $\mathfrak{g}_Q$ - the Kac-Moody algebra with $Q$ as its Dynkin diagram, specialized at $\nu=\sqrt{q}$.
When $\mc{A} = Coh(\mathbb{P}^1_{\mathbb{F}_q})$, Kapranov \cite{K1} and Baumann-Kassel \cite{BK} show that there is an embedding $U^+_{\nu}(\widehat{\mathfrak{sl}}_2) \hookrightarrow \wt{H}_{\mc{A}}$, where the former is a certain positive subalgebra of the quantum affine algebra of $\mathfrak{sl}_2$. The Hall algebras of elliptic curves have been studied by Burban-Schiffmann, Schiffmann, and Schiffmann-Vasserot \cite{BS, S2, SV}, and shown to be connected with DAHA (Double affine Hecke algebras) as well as a number of other representation-theoretic objects. Hall algebras of higher-genus projective curves were studied in \cite{KSV}. 

There is an extensive body of work connecting Hall algebras, geometric representation theory, the study of moduli spaces of sheaves/quiver representations, and the Langlands program which we cannot possibly summarize here. We mention only that understanding variants of the Hall algebra $\wt{H}_{Coh(X)}$ when $dim(X) > 1$ is an important and difficult problem (see for instance the work of Kapranov-Vasserot \cite{KV, KV2} ). 

\subsection{Hall algebras in the non-additive setting}

An examination of the product (\ref{hall_multiplication})  reveals that the requirement that $\mc{A}$ be abelian or even additive is unnecessary. In \cite{DK} Dyckerhoff-Kapranov introduce \emph{proto-exact} categories. These are generalizations of Quillen exact categories which satisfy a minimal set of axioms designed to yield an associative product via (\ref{hall_multiplication}) whenever they are finitary. 

Many non-additive examples of proto-exact categories come from combinatorics. Here, $\mc{A}$ typically consists of combinatorial structures equipped an operation of "collapsing" a sub-structure, which corresponds to forming a quotient in $\mc{A}$. Examples of such $\mc{A}$ include trees, graphs, posets, matroids, semigroup representations on pointed sets, quiver representations in pointed sets etc. (see \cite{EJS, KS, Sz1, Sz2, Sz3, Sz4} ). The product in $\H_{\mc{A}}$, which counts all extensions between two objects, thus amounts to enumerating all combinatorial structures that can be assembled from the two. In this case, $\H_{\mc{A}}$ is (dual to) a combinatorial Hopf algebra in the sense of \cite{LR}. Many combinatorial Hopf algebras arise via this mechanism. 

\subsection{$\fun$, Hall algebras, and classical limits}

It's an old observation that many combinatorial problems over $\mathbb{F}_q$ have well-defined limits as $q \rightarrow 1$, in which they reduce to analogous problems about finite sets. This is well illustrated by the example of the Grassmannian, where the number of $\mathbb{F}_q$-points is given by a $q$-binomial coefficient:
\[
\# \on{Gr}(k,n)/\mathbb{F}_q = \frac{[n]_q !}{[n-k]_q ! [k]_q !} 
\]
with
\[
[n]_q ! = [n]_q [n-1]_q \ldots [2]_q
\]
and 
\[
[n]_q = 1 + q + q^2 + \ldots + q^{n-1}.
\]
In the limit $q \rightarrow 1$ this reduces to the binomial coefficient $\binom{n}{k}$, counting $k$-element subsets of an $n$-element set. Another famous observation due to Tits is that if $G$ is a simple algebraic group of adjoint type, and we denote by $|G(\mathbb{F}_q)|$ the number of points of $G$ over $\mathbb{F}_q$,  then (suitably normalized) $$ \lim_{q \rightarrow 1} |G(\mathbb{F}_q)| = | W(G)|,$$
where $W(G)$ is the Weyl group of $G$. 

Finding a meaningful unifying framework for these and other observations has led to attempts to define mathematics over $\fun$ - "the field of one element". We refer the interested reader to the review \cite{lorscheid2018f1}. In its most pedestrian form, the yoga of $\fun$ states that a vector space over $\fun$ is a pointed set, and monoids are the analogues of algebras. These structures are manifestly "non-additive". One version of algebraic geometry over $\fun$ developed by Deitmar \cite{D1,D2,D3}, Kato \cite{Kato}, Soul\'e \cite{soule2004varieties}, Connes-Consani-Marcolli \cite{CC1,CC2,CCM}, and Cortinas-Haesemayer-Walker-Weibel \cite{CHWW}, is the theory of  \emph{monoid schemes}. These are spaces glued from spectra of commutative monoids and equipped with a structure sheaf of monoids, much as ordinary schemes are glued from spectra of commutative rings. 

One may ask how the $q \rightarrow 1$ limit behaves at the level of Hall algebras, when such a comparison makes sense. For instance, when $Q$ is a quiver, we can try to compare the Hall algebras of the categories $Rep(Q, \mathbb{F}_q)$ and $Rep(Q, \fun)$. Similarly, if $X$ is a projective toric variety (which in particular implies that it arises via base change from a monoid scheme), we may try to compare the Hall algebras of $Coh(X_{\mathbb{F}_q})$ and $Coh(X_{\fun})$. The results in \cite{Sz4, Sz1} suggest that the Hall algebras computed in the $\fun$ world are (possibly degenerate) $q \rightarrow 1$ (i.e. "classical")  limits of their $\mathbb{F}_q$  analogues. By doing calculations in the $\fun$ world we may thus hope to learn something about the Hall algebras of $Coh(X)$ when $dim(X) > 1$, where  little is known.  

A rationale for why computations over $\fun$ are related to classical limits is explored in \cite{CLMT}. For instance, a two-dimensional vector space over a field may be viewed as Qbit, where superposition of quantum states is allowed, whereas a two-dimensional vector space over $\fun$ corresponds to a classical bit, the two elements corresponding to $0, 1$. 

\subsection{Summary of main results}

Let $\Delta \subset \mathbb{R}^n$ be a fan (in the sense of toric geometry). Each cone $\sigma \in \Delta$ defines a commutative monoid $S_{\sigma}$, and $\Delta$ provides data for gluing their spectra, to obtain a monoid scheme $(X_{\Delta}, \mc{O}_X)$. 
The category of coherent sheaves $\Coh(X)$, defined in a manner analogous to ordinary algebraic geometry, is now non-additive. Paraphrased, our first result is:

\begin{inthm}[Proposition \ref{coh_pe_thm}]
$\Coh(X)$ has the structure of a proto-exact category in the sense of Dyckerhoff-Kapranov. 
\end{inthm}

This means that one can define exact sequences and $\on{Ext}$, and that $\Coh(X)$ has an associated algebraic K-theory defined via the Waldhausen construction. There is however an obstacle to defining the Hall algebra $\H_{\Coh(X)}$, since the category $\Coh(X)$ is not finitary even when $\Delta$ is the fan of a smooth projective toric variety such as say $\mathbb{P}^n$. To overcome this difficulty we define subcategories
\[
\Coh^T(X) \subset \Coh^{\alpha}(X) \subset \Coh(X)
\]
The subcategory $\Coh^{\alpha}(X)$ consists of coherent sheaves $\F$ having the property that the action of the monoid structure sheaf $\mc{O}_X$ on $\F$ is sufficiently nice, and contains all locally free sheaves. More precisely, we require that multiplication by a section of the structure sheaf not identify two distinct sections of $\F$ unless these get sent to $0$. This turns out to eliminate the problem of non-finitarity and allows us to define a Hall algebra:

\begin{inthm}[Theorem \ref{finitary_alpha}] 
Let $\Delta$ be the fan of a projective toric variety, and $(X_{\Delta}, \mc{O}_X)$ the corresponding monoid scheme. Then
\begin{enumerate}
\item $\Coh^{\alpha}(X_{\Delta})$ is a finitary proto-abelian category.  
\item The Hall algebra $H_{\Coh^{\alpha}(X_{\Delta})}$ is isomorphic, as a Hopf algebra, to an enveloping algebra $U(\mathfrak{n}^{\alpha}_X)$, where $\mathfrak{n}^{\alpha}_X$ has the indecomposable sheaves in $\Coh^{\alpha}(X_{\Delta})$ as a basis. 
\end{enumerate}
\end{inthm}

We note that the standard cohomological machinery available for abelian categories of coherent sheaves on ordinary schemes has yet to be developed in the $\fun$ world. We prove the above theorem by identifying $\on{Ext}(\F, \F')$ for $\F, \F' \in \Coh^{\alpha}(X)$ with a subset of $\on{Ext}^1_{\Coh(X_{\Delta,k})} (\F_k, \F'_k)$ where $k$ is a finite field, $X_{\Delta,k}$ denotes the ordinary toric variety over $k$ associated to $\Delta$, and $\F_k$ the scalar extension ($k$-linearization) of the sheaf $\F$. 

We may further refine the class of coherent sheaves under consideration by passing to a subcategory $\Coh^T(X) \subset \Coh^{\alpha}(X)$ of sheaves which locally (on affines) are of the form $\oplus^k_{i=1} \mc{I}_i/\mc{J}_i$, where $\mc{J}_i \subset \mc{I}_i \subset \mc{O}_X$ are ideal sheaves. When $X$ is smooth (in particular, each $S_{\sigma} \simeq \mathbb{Z}^n_{\geq 0}$ ), such sheaves are shown to  be glued from "skew shapes". Here by a skew shape we mean a convex sub-poset of $\mathbb{Z}_{\geq 0}^n$ under the partial order where $(x_1, \cdots, x_n) \geq (y_1, \cdots, y_n)$  if and only if $x_i \leq y_i $ for $i=1, \cdots, n$. For example, when $n=2$, the infinite shape

\begin{center}
\Ylinethick{1pt}
\gyoung(:\vdots:\vdots,;;,;;,;;,\bullet;;;;:\cdots,:;\bullet;;;:\cdots,::;\bullet;;:\cdots)
\end{center}
\Ylinethick{1pt}

carries an action of the free commutative monoid $\fun \langle x_1, x_2 \rangle$ generated by $x_1, x_2$,  with $x_1$ moving one box to the right, and $x_2$ one box up, until the edge of the diagram is reached, and $0$ beyond that. This yields a module on three generators (indicated by black dots),  which may be interpreted as describing a coherent sheaf on $\mathbb{A}^2$ supported on the union of the $x_1$ and $x_2$-axes. Similarly, the three-dimensional infinite plane partition 
\begin{center}
\includegraphics[scale=0.8]{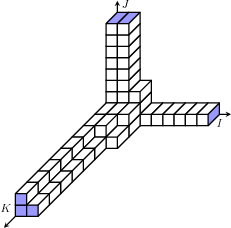}
\end{center}
describes an object of $\Coh^T(\mathbb{A}^3)$ supported on the union of the three coordinate axes. 

Since $\Coh^T(X)$ is a full subcategory of $\Coh^{\alpha}(X)$, it is finitary whenever $\Delta$ is the fan of a projective toric variety, and we may consider its Hall algebra $H^T_X$. We show:

\begin{inthm}[Theorem \ref{finitary_T}]
Let $\Delta$ be the fan of a projective toric variety, and $(X_{\Delta}, \mc{O}_X)$ the corresponding monoid scheme. Then
\begin{enumerate}
\item $\Coh^T(X_{\Delta})$ is a finitary proto-abelian category.  
\item The Hall algebra $H_{\Coh^T(X_{\Delta})}$ is isomorphic, as  a Hopf algebra, to an enveloping algebra $U(\mathfrak{n}^{T}_X)$, where $\mathfrak{n}^{T}_X$ has the indecomposable sheaves in $\Coh^{T}(X_{\Delta})$ as a basis. 
\end{enumerate}

\end{inthm}

In addition to $\Coh^T(X)$, we also consider subcategories $\Coh^T(X)_Z \subset \Coh^T(X)$ of sheaves supported in a closed subset $Z \subset X$, and $\Coh^T(X)_{\mc{I}} \subset \Coh^T(X)$ of sheaves scheme-theoretically supported in a closed subscheme determined by a quasicoherent sheaf of ideals $\mc{I} \subset \mc{O}_X$. These are also proto-abelian, and finitary whenever $\Coh^T(X)$ is, allowing us to define Hall algebras. 

The last section of the paper is devoted to computing the Hall Lie algebra $\mathfrak{n}^T_X$ for various monoid schemes $X$, and whenever possible, relating these to previously studied Lie algebras such as loop algebras.  In non-trivial examples, the Lie algebras $\n^T_X$ are large infinite-dimensional Lie algebras whose complexity grows rapidly with $dim(X)$. 

\begin{itemize}
\item When $X=\mathbb{P}^1$, $\n^T_X$ was computed by the second author in \cite{Sz1}, where the problem of Hall algebras of monoid schemes was first considered. It was shown to be isomorphic to a non-standard Borel subalgebra in $\mathfrak{gl}_2 [t,t^{-1}]$.
\item The Hall algebra of the category $\Coh^T(\mathbb{A}^n)_0$ of point sheaves supported at the origin in $\mathbb{A}^n$ was studied in \cite{Sz0} and described in terms of a Lie algebra of skew shapes. 
\item We compute the Hall algebra of the subcategory of $\Coh^T(\mathbb{A}^2)$ consisting of sheaves supported on the second infinitesimal neighborhood of the origin, and show that  its Hall Lie algebra embeds into $\mathfrak{gl}_2 [t]$. 
\item We consider a certain subcategory of $\Coh^T(\mathbb{A}^2)_0$ generated by skew shapes having at most two rows and show its Hall Lie algebra embeds into $\mathfrak{gl}_{\infty}[t]$.
\item We produce an explicit basis for the Hall Lie algebra of $\Coh^T(\mathbb{P}^2)$ by classifying all indecomposable $T$-sheaves on $\mathbb{P}^2$, and show that the Lie subalgebra generated by line bundles supported on the triangle of  $\mathbb{P}^1$'s surjects onto $\wt{\mathfrak{gl}}^-_{\infty}$, where the latter denotes lower-triangular $\mathbb{Z} \times \mathbb{Z}$ matrices. 
\end{itemize}

\subsection{Acknowledgements}
J.J. was supported by an AMS-Simons travel grant. M.S. is grateful to Olivier Schiffmann for many valuable conversations and suggestions, and to Tobias Dyckerhoff for answering several questions regarding the formalism of proto-exact/proto-abelian categories.  He also gratefully acknowledges the support of a Simons Foundation Collaboration Grant during the writing of this paper. 

\section{Proto-exact categories and their Hall algebras}

In this section we review the notions of proto-exact and proto-abelian categories and their Hall algebras introduced in \cite{DK, Dy}. 

\subsection{Proto-exact and proto-abelian categories} \label{pe_categories}

A commutative square 
\begin{equation}\label{comm_square}
\begin{tikzcd}
X
\ar["i",r]
\ar["j",d,swap]
&
Y
\ar["j'",d]
\\
X'
\ar["i'",r,swap]
&
Y'
\end{tikzcd}
\end{equation}
is said to be biCartesian if it is both Cartesian and co-Cartesian.

\begin{mydef} \label{pe_def}
	A {\em proto-exact} category is a pointed category $\mathcal{C}$ equipped with two special classes of morphisms $\mathfrak{M}$ and $\mathfrak{E}$, called \emph{admissible monomorphisms} and \emph{admissible epimorphisms} respectively. The triple $(\mathcal{C},\mathfrak{M},\mathfrak{E})$ is required to satisfy the following properties.
	\begin{enumerate}
		\item 
		Any morphism $0 \to A$ is in $\mathfrak{M}$ and any morphism $A \to 0$ is in $\mathfrak{E}$.
		\item 
		The classes $\mathfrak{M}$ and $\mathfrak{E}$ are closed under composition and contain all isomorphisms.
		\item 
		A commutative square \eqref{comm_square} in $\C$ with
		$i,i' \in \MM$ and $j, j' \in \EE$ is Cartesian if and only if it is
		co-Cartesian.
		\item 
		Every diagram of the following form
		\[
		\begin{tikzcd}
		X
		\ar[r,"i",hook]
		\ar[d,"j",two heads,swap]
		&
		Y
		\\
		X'
		&
		\end{tikzcd}
		\]
		with $i \in \MM$ and $j \in \EE$
		can be completed to a biCartesian square \eqref{comm_square} with $i' \in \MM$ and $j' \in \EE$.
		
		\item 
		Every diagram of the following form
		\[
		\begin{tikzcd}
		&
		Y
		\ar[d,two heads,"j'"]
		\\
		X'
		\ar[r,hook,"i'",swap]
		&
		Y'
		\end{tikzcd}
		\]
		with $i' \in \MM$ and $j' \in \EE$
		can be completed to a biCartesian square \eqref{comm_square} with $i \in \MM$ and $j \in \EE$.
	\end{enumerate}
\end{mydef}

\begin{rmk}
	If $\MM$ is the class of all monomorphisms in $\mathcal{C}$ and $\EE$ is the class of all epimorphisms in $\mathcal{C}$, then $\mathcal{C}$ is said to be a \emph{proto-abelian} category \cite{Dy}.
\end{rmk}

A biCartesian square of the following form
\[
\begin{tikzcd}
X
\ar[r,hook,"i"]
\ar[d,two heads,"j",swap]
&
Y
\ar[d,two heads,"j'"]
\\
0
\ar[r,hook,"i'",swap]
&
Z
\end{tikzcd}
\]
with $i,i' \in \MM$ and 
$j.j' \in \EE$ is said to be an \emph{admissible short exact sequence} or an \emph{admissible
	extension of $Z$ by $X$}, and will also be denoted
\begin{equation}\label{ses}
\ses{X}{Y}{Z}.
\end{equation}
In this case we will frequently denote the object $Z$ (unique up to isomorphism) by $Y/X$.

Two extensions $\ses{X}{Y}{Z}$ and $\ses{X}{Y'}{Z}$ of
$Z$ by $X$ are \emph{equivalent} if there is a commutative diagram as follows:
\begin{equation} \label{cd_ses}
\begin{tikzcd}
X
\ar[r,hook]
\ar["\on{id}",d,swap]
&
Y
\ar[r,two heads]
\ar["\cong",d]
&
Z
\ar["\on{id}",d]
\\
X
\ar[r,hook]
&
Y'
\ar[r,two heads]
&
Z
\end{tikzcd}
\end{equation}
The set of equivalence classes of such sequences is denoted
$\on{Ext}_{\C}(Z,X)$. We will assume that $\C$ has the following additional properties:  

\begin{enumerate}
\item $\C$ has finite coproducts, which we denote by $X \oplus Y$. This, together with the fact that $\C$ is pointed, implies that there exist morphisms $\pi_X : X \oplus Y \rightarrow X $  and $\pi_Y : X \oplus Y \rightarrow Y $  such that the composition 
\[
X \longmapsto  X \oplus Y \overset{\pi_X}{\longmapsto} X
\]
is $\id_X$, and the composition 
\[
X \longmapsto  X \oplus Y \overset{\pi_Y}{\longmapsto} Y
\]
is $0$. 
\item The map $X \rightarrow X \oplus Y$ is in $\mathfrak{M}$, and $\pi_X \in \mathfrak{E}$
\end{enumerate}

A short exact sequence equivalent to one of the form $$ \ses{X}{X \oplus Y}{Y} $$ will be called a \emph{split exact sequence}. 

\begin{mydef}
Two admissible monomorphisms
$i_1 \colon X \hookrightarrow Y$ and $i_2 \colon X' \hookrightarrow Y$
are \emph{isomorphic} if there is an isomorphism
$f \colon X \rightarrow X'$ with $i_1= i_2 \circ f$.  The isomorphism
classes in $\MM$ are \emph{admissible sub-objects}.
\end{mydef}

\begin{mydef}
A functor $F \colon \C \to \mathcal{D}$ between proto-exact
categories is \emph{exact} when it preserves admissible short exact
sequences.  
\end{mydef}

\begin{mydef}\label{defi:finitary}
	A proto-exact category $(\mathcal{C},\MM,\EE)$ is {\em finitary} if, for every pair of
	objects $X$ and $Y$, the sets $\on{Hom}_{\C}(X,Y)$ and
	$\on{Ext}_{\C}(X,Y)$ are finite.
\end{mydef}

\begin{myeg}\label{example: proto-exact}
	The following are examples of proto-exact categories.
	\begin{enumerate}
		
		\item Any Quillen exact category is proto-exact, with the same exact
		structure. In particular, any abelian category $\C$ is proto-exact
		with $\MM$ all monomorphisms and $\EE$ all epimorphisms
		respectively. The category $\on{Rep}(Q, \mathbb{F}_q)$ of
		representations of a quiver $Q$ over a finite field $\mathbb{F}_q$
		and the category $\on{Coh}(X / \mathbb{F}_q)$ of coherent sheaves on a 
		projective variety $X$ over $\mathbb{F}_q$ are both finitary abelian.
		
		\item The category of pointed sets and pointed maps, with $\MM(X,Y)$ all pointed injections $X \hookrightarrow Y$, and $\EE(X,Y)$ all pointed surjections $f: X \twoheadrightarrow Y $ satisfying
		\begin{equation}  \label{spec_condition}  
		f \vert_{X \backslash f^{-1}(*)} \textrm{ is an injection }, \end{equation} is proto-exact. The full
		subcategory $\Set^{fin}$ of finite pointed sets is finitary. The category $\vfun$ of finite pointed sets together with all pointed maps satisfying the condition \eqref{spec_condition} is proto-abelian and finitary. It is a non-full subcategory of $\Set^{fin}$. 
		
		\item The category $\MS$ of pointed matroids and strong maps is finitary proto-exact \cite{EJS}.
		
		\item Suppose $\I$ is a small category, $\C$ proto-exact,  and $\on{Fun}(\I, \C)$ the category of functors from $\I$ to $\C$, with natural transformations as morphisms. $\on{Fun}(\I, \C)$ is proto-exact (see \cite{DK}), with 
		\begin{itemize} 
			\item $\MM(F,G) = \{ \textrm{ natural transformations } \phi: F \Rightarrow G \mid  \phi_X : F(X) \rightarrow G(X) \in \MM(F(X),G(X)) \} $
			\item $\EE(F,G) = \{ \textrm{ natural transformations } \phi: F \Rightarrow G \mid \phi_X : F(X) \rightarrow G(X) \in \EE(F(X),G(X)) \} $
		\end{itemize} 
		If $\C$ is proto-abelian, then $\on{Fun}(\I, \C)$ is proto-abelian. 
		
		\item As a special case of the previous example, a monoid $\A$ (see Section \ref{monoid_sec}) can be viewed as a category with one object $\{ \star \}$, and $\A = \on{Hom}( \star, \star)$. The category $\on{Fun}(A, \Set^{fin})$ is then equivalent to the category $\Amod$ of $\A$-modules (or $\A$-acts), whose objects are finite pointed sets with an action of $\A$. By the previous example, $\Amod$ is proto-exact, where for $M, N \in \Amod$
		\begin{itemize}
			\item $\MM(M,N) = \{ \phi \in \on{Hom}_{\Amod} (M,N) \mid \phi \textrm{ is injective } \}$
			\item $\EE(M,N) =  \{ \phi \in \on{Hom}_{\Amod} (M,N) \mid \phi \textrm{ is a surjection satisfying \eqref{spec_condition}} \}$
		\end{itemize}
		Similarly, the category $\on{Fun}(\A, \vfun)$ is proto-abelian, and may be identified with the category of \emph{type}-$\alpha$ modules $\Amod^{\alpha}$ in Section \ref{subsection: type=alpha modules}. Specializing further, denoting by $P(Q)$ the path monoid of a quiver $Q$, we obtain a proto-abelian category $\on{Rep}(Q, \fun) := \on{Fun}(P(Q), \vfun)$ of quiver representations over $\fun$ considered in \cite{Sz4}. 
	\end{enumerate}
\end{myeg}

\subsection{Hall algebras of finitary proto-exact categories}\label{subsection: Hall algebras of finitary proto-exact}

Let $\C$ be a finitary proto-exact category, and $k$ a field of
characteristic zero. Define the Hall algebra $\H_{\C}$ over $k$ as
\[
  \H_{\C}
  \coloneqq
  \{ f \colon \on{Iso}(\C) \rightarrow k \mid f
    \text{ has finite support} \},
\]
where $\on{Iso}(\C)$ denotes the set of isomorphism classes in
$\C$. The Hall algebra $\H_{\C}$ is an associative $k$--algebra under
the convolution product
\begin{equation}\label{hall_product}
  f \bullet g ([X])
  \coloneqq
  \sum_{X' \subset X} f([X/X'])g([X']),
\end{equation}
where the summation $\sum_{X' \subset X}$ is taken over isomorphism
classes of admissible sub-objects
$i \colon X' \hookrightarrow X, i \in \MM$, and $[-]$ denotes
isomorphism classes in $\C$.  The algebra $\H_{\C}$ has a basis of $\delta$-functions
$\set{\delta_{[X]}}{[X] \in \on{Iso}(\C)}$, where
\[
  \delta_{[X]} ([X'])
  =
  \begin{cases} 
    1 & X' \simeq X, \\
    0 & \textrm{ otherwise. }
  \end{cases}
 \]
The multiplicative unit of $\H_{\C}$ is $\delta_{[0]}$.  The structure
constants of this basis are given by
\[
  \delta_{[X]} \bullet \delta_{[Y]}
  =
  \sum_{[Z] \in \on{Iso}(\C)} g^Z_{X,Y} \delta_{[Z]},
\]
where
\[
  g^Z_{X,Y} = \#\{Z' \subset Z \mid Z'\simeq Y,Z/Z'\simeq X\}.
\]
Thus $g^Z_{X,Y}$ counts the number of admissible subobjects $Z'$ of $Z$
isomorphic to $Y$ such that $Z/Z'$ is isomorphic to $X$. 

Whether $\H_{\C}$ carries a co-multiplication compatible with \eqref{hall_product} 
depends on further properties of $\C$.  If $\C$ is finitary, abelian,
linear over $\mathbb{F}_q$, and hereditary, then $\H_{\C}$ carries the
so-called \emph{Green's co-multiplication} (see \cite{G}). We will be
concerned with situations where $\C$ is not additive, but where a simpler alternative construction applies.
To this end we assume that $\C$ has the following additional property:

\begin{enumerate}
\item[(3)] The only admissible sub-objects of $X \oplus Y$ are of the form $X' \oplus Y'$,
where $X' \subset X, Y' \subset Y$.
\end{enumerate}

We note that the categories $\Set$ and $\MS$ satisfy properties above $(1)-(3) $, but that $(3)$ generally fails for abelian categories. 
Define
\begin{equation}\label{coprod-def}
  \Delta \colon \H_{\C} \rightarrow \H_{\C} \otimes \H_{\C},
  \;
  \Delta(f)([X], [Y])\rightarrow f([X \oplus Y])
  .
\end{equation}
Under the assumptions $(1) - (3)$ on $\C$, $\Delta$ is easily seen to equip $\H_{\C}$ with a bialgebra structure. 
$\Delta$ is co-commutative by \eqref{coprod-def}, since $[X \oplus Y]=[Y \oplus X]$, and
the subspace of primitive elements (i.e. those satisfying $\Delta(x)=x\otimes 1 + 1\otimes x$) of $\H_{\C}$ is spanned by the set
of all $\delta_{[X]}$ for $X$ indecomposible (i.e. those $X$'s which cannot be
written as a non-trivial co-product).  
Furthermore, $\H_{\C}$ is naturally graded by $K_0 (\C)^+  \subset K_0 (\C)$, where  $K_0 (\C)^+$ denotes
the sub-semigroup generated by the effective classes, with $\deg(\delta_{[X]}) = [X] \in K_0 (\C)^+$. 
As any graded, connected, and co-commutative bialgebra is a Hopf algebra isomorphic to the enveloping algebra of the Lie algebra of its primitive elements by the Milnor-Moore theorem. We thus have the following(\cite{EJS}):
\begin{prop}\label{Hall_theorem}
  Let $\C$ be a finitary proto-exact category $\C$ satisfying
  additional properties $(1)-(3)$ above. 
  Then $\H_{\C}$ has the structure of a $K^+_0(\C)$-graded, connected,
  co-commutative Hopf algebra over $k$ with co-multiplication
  (\ref{coprod-def}).  Moreover
  $\H_{\C}$ is isomorphic as a Hopf algebra to the enveloping algebra  $ \mathbf{U}(\n_{\C})$, where $\n_{\C}$ is a $K^+_0(\C)$-graded Lie algebra with basis $\delta_{{[X]}}$, with $X$ indecomposable.
\end{prop}

We will refer to $\n_{\C}$ as the \emph{Hall Lie algebra of $\C$}.

\begin{myeg}
Consider the category $\C := \vfun$. Coproducts in $\vfun$ correspond to wedge sums, and $\vfun$ is easily seen to satisfy the conditions of Proposition \ref{Hall_theorem}. 
 We have $\H_{\C} \simeq k[x]$, with
  \[
    \Delta(x) = x\otimes 1 + 1 \otimes x,
   \]
   where $x=\delta_{\{e,*\}}$ is the delta-function supported on the pointed set with one non-zero element $e$. 
\end{myeg}

\begin{myeg} Let $\langle t \rangle$ denote the free monoid on the generator $t$, and let $\C$ denote the category of nilpotent $\langle t \rangle$-modules. Indecomposable objects of $\C$ may be identified with rooted trees (see Example \ref{alphaA1}). $\H_{\C}$ is isomorphic to the dual of the Connes-Kreimer Hopf algebra of rooted trees \cite{CK, Sz2}.
\end{myeg}

\begin{myeg}
Let $Q_0$ be the Jordan quiver, and let $\C=\textrm{Rep}(Q_0,\mathbb{F}_1)_{\textrm{nil}}$ be the category of nilpotent representations of $Q_0$ over $\mathbb{F}_1$ (see \cite{Sz4}). The Hall algebra $H_\C$ is isomorphic as a Hopf algebra to the ring $\Lambda$ of symmetric functions. Under this isomorphism, the indecomposible objects in $\C$ correspond to the power sum symmetric functions in $\Lambda$.  
\end{myeg}

\begin{myeg}
When $\C=\MS$, $\H_{\C}$ is isomorphic to the dual of Schmitt's matroid-minor Hopf algebra (\cite{Sch, EJS}). 
\end{myeg}

\section{Monoid schemes} \label{msch}

In this section, we briefly review notions pertaining to monoid schemes and quasicoherent sheaves on them. The theory of monoid schemes was developed by Kato \cite{Kato}, Deitmar \cite{D1, D2, D3}, Connes-Consani-Marcolli \cite{CC1, CC2, CCM}, and Cortinas-Haesemayer-Walker-Weibel \cite{CHWW}. We refer the interested reader to  the first few sections of \cite{CHWW} for a concise summary.  

\subsection{Monoids and \textrm{MSpec}} \label{monoid_sec}

Recall that ordinary schemes are ringed spaces locally modeled on affine schemes, which are spectra of commutative rings.  A monoid scheme is locally modeled on an affine monoid scheme, which is the spectrum of a commutative unital  monoid with $0$. In the following, we will denote monoid multiplication by juxtaposition or $\cdot$. In greater detail:

A \emph{monoid} $A$ will be a commutative associative monoid with identity $1_A$ and zero $0_A$ (i.e. the absorbing element). We require
\[
 1_A \cdot a = a \cdot 1_A = a, \hspace{1cm} 0_A \cdot a = a \cdot 0_A = 0_A, \hspace{1cm} \forall a \in A.
\]
Maps of monoids are required to respect the multiplication as well as the special elements $1_A, 0_A$. An \emph{ideal} of $A$ is a subset $\a \subset A$ such that $\a \cdot A \subset \a$. A proper ideal $\p \subset A$ is \emph{prime} if $xy \in \p$ implies either $x \in \p$ or $y \in \p$, and \emph{maximal} if it is not properly contained in another proper ideal. Maximal ideals are prime.  A monoid $A$ is:
\begin{itemize}
 \item \emph{cancellative} if for $a, b, c \in A$, and $a \neq 0$, $ab = ac$ implies $b=c$
 \item  (following \cite{HW}) \emph{partially cancellative} if $ab=ac \neq 0$ implies $b=c$
 \end{itemize}

Given a monoid $A$, the topological space $\spec A$ is defined to be the set $$\spec A := \{ \p \mid \p \subset A \textrm{ is a prime ideal } \}, $$ with the closed sets of the form $$ V(\a) := \{ \p \mid \a \subset \p, \p \textrm{ prime } \}. $$ 
Given a multiplicatively closed subset $S \subset A$, the \emph{localization of $A$ by $S$}, denoted $S^{-1}A$, is defined to be the monoid consisting of symbols $$\{ \frac{a}{s} \mid a \in A, s \in S \},$$ with the equivalence relation $$\frac{a}{s} = \frac{a'}{s'}  \iff \exists \; s'' \in S \textrm{ such that } as's'' = a's s'', $$ and multiplication is given by $\frac{a}{s} \times \frac{a'}{s'} = \frac{aa'}{ss'} $. 

For $f \in A$, let $S_f$ denote the multiplicatively closed subset $\{ 1, f, f^2, f^3, \cdots  \}$. We denote by $A_f$ the localization $S^{-1}_f A$, and by $D(f)$ the open set $\spec A \backslash V(f) \simeq \spec A_f$, where $V(f) := \{ \p \in \spec A \mid f \in \p \}$. The open sets $D(f)$ form a basis for $\spec A$. 
$\spec A$ is equipped with a  \emph{structure sheaf} of monoids $\mc{O}_A$, satisfying the property $\Gamma(D(f), \mc{O}_A) = A_f$. Its stalk at $\p \in \spec A$ is $A_{\p} := S^{-1}_{\p} A$, where $S_{\p} = A \backslash \p$. 

A unital homomorphism of monoids $\phi: A \rightarrow B$ is \emph{local} if $\phi^{-1}(B^{\times}) \subset A^{\times}$, where $A^{\times}$ (resp.~$B^{\times}$) denotes the invertible elements in $A$ (resp.~$B$). 
A \emph{monoidal space} is a pair $(X, \mc{O}_X)$ where $X$ is a topological space and $\mc{O}_X$ is a sheaf of monoids. A \emph{morphism of monoidal spaces} is a pair $(f, f^{\#})$ where $f: X \rightarrow Y$ is a continuous map, and $f^{\#}: \mc{O}_Y \rightarrow f_* \mc{O}_X$ is a morphism of sheaves of monoids, such that the induced morphism on stalks $f^{\#}_\p : \mc{O}_{Y, f(\p)} \rightarrow f_* \mc{O}_{X, \p}$ is local. 
An \emph{affine monoid scheme} is a monoidal space isomorphic to $(\spec A, \mc{O}_A)$. 
Thus, the category of affine monoid schemes is opposite to the category of monoids. 
A monoidal space $(X,\mc{O}_X)$ is called a \emph{monoid scheme}, if for every point $x \in X$ there is an open neighborhood $U_x \subset X$ containing $x$ such that $(U_x, \mc{O}_X \vert_{U_x})$ is an affine monoid scheme. We denote by $\Msch$ the category of monoid schemes, with morphisms those of monoidal spaces. 

A monoid scheme is said to be \emph{partially cancellative} if it can be covered by spectra of partially cancellative monoids, and of \emph{finite type} if it has a finite covering by spectra of finitely generated monoids. Going forward, we make the following assumption:
\begin{center}
{\bf All monoid schemes appearing in this paper will be of finite type and partially cancellative}
\end{center}

\begin{myeg}
Let $\fun=\{0,1\}$, with multiplication defined by
\[
1\cdot 1=1, \quad 1 \cdot 0 = 0\cdot 1 =  0 \cdot 0 = 0.
\]
One can easily see that $\fun$ is the initial object in the category of monoids. $\fun$ is sometimes referred to as \emph{the field with one element}. 
\end{myeg}

The following example plays a role of the polynomial rings in our setting. 

\begin{myeg} \label{freecom}
Let $\fun\langle x_1,\dots,x_n\rangle:=\{x_1^{k_1}x_2^{k_2}\cdots x_n^{k_n} \mid k_i \in \mathbb{Z}_{\geq 0}\} \cup \{0\}$ be the set of monomials in variables $x_1,\dots ,x_n$ along with an extra element $0$. With the usual multiplication of monomials and $0\cdot x_1^{k_1}x_2^{k_2}\cdots x_n^{k_n}=x_1^{k_1}x_2^{k_2}\cdots x_n^{k_n}\cdot 0 =0$, $\fun\langle x_1,\dots,x_n\rangle$ is a monoid. We will frequently use the multi-index notation to write elements in $\fun\langle x_1,\dots,x_n\rangle$. For instance, with $a=(a_1,\dots,a_n) \in \mathbb{Z}_{\geq 0}^n$, we write $x^a:=x_1^{a_1}x_2^{a_2}\cdots x_n^{a_n}$.  
\end{myeg}

\begin{myeg} \label{example: P1}

Let 
\[
\langle t \rangle := \fun \langle t \rangle  = \{ 0, 1, t, t^2, \cdots, t^n, \cdots \},
\]
and let $\mathbb{A}^1 := \spec \, \langle t \rangle $ - \emph{the monoid affine line}.
Let $\langle t, t^{-1} \rangle$ be the following monoid 
\[
\langle t,t^{-1} \rangle := \{ \cdots, t^{-2}, t^{-1}, 1, 0, t, t^2, \cdots \}.
\]
We obtain the following diagram of inclusions
\begin{equation*}
\langle t \rangle \hookrightarrow \langle t,t^{-1} \rangle \hookleftarrow \langle t^{-1} \rangle.
\end{equation*}
By taking spectra, and denoting by $U_0 = \spec \, \langle t \rangle , U_{\infty} = \spec \, \langle t^{-1} \rangle $, we obtain
the following diagram
\begin{equation*} 
\mathbb{A}^1 \simeq U_0 \hookleftarrow U_0 \cap U_{\infty} \hookrightarrow U_{\infty} \simeq \mathbb{A}^{1}.
\end{equation*}
We define $\mathbb{P}^{1}$, the \emph{monoid projective line}, to be the
monoid scheme obtained by gluing two copies of $\mathbb{A}^1$
according to the maps in the above diagram. 
$\mathbb{P}^{1}$ has three points - two closed points $0 \in U_0, \;  \infty \in U_{\infty}$, and the generic point $\eta$. We denote the corresponding inclusions by $\iota_{0} : U_0 \hookrightarrow \Pone$, $\iota_{\infty}: U_{\infty} \hookrightarrow \Pone$. 

\end{myeg}

\subsection{Coherent sheaves} \label{coh_sheaves}




\subsubsection{$A$-modules} \label{Amod}

Let $A$ be a monoid. An \emph{$A$--module} is a pointed set $(M,*_{M})$ together with an action 
\[
\mu:A \times M \to M, \quad (a,m) \rightarrow a\cdot m
\]
which is compatible with the monoid multiplication; $1_A \cdot m = m$, $a \cdot (b \cdot m) = (a \cdot b) \cdot m$, and $0_A \cdot m = *_{M} \; \forall m \in M$. We will refer to  $M \backslash *_M$ as \emph{non-zero} elements, and to $*_M$ as the zero element, sometimes denoting it by $0$. 
A \emph{morphism of $A$--modules} $f: (M,*_M) \rightarrow (N,*_N)$ is a map of pointed sets (i.e. we require $f(*_M) = *_N$) which is compatible with the action of $A$; $f(a \cdot m) = a \cdot f(m)$. $Hom_A(M,N)$ is an $A$-module via the action $(a \cdot f) (m) := a \cdot (f(m))$.  We denote by $\Amod$ the category of $A$-modules.

A pointed subset $(M',*_M) \subset (M,*_M)$ is called an \emph{$A$--submodule} if $A \cdot M' \subset M'$. In this case we may form the quotient module $M/M'$, where $M/M' := M \backslash (M' \backslash *_M)$, $*_{M/M'} = *_M$, and the action of $A$ is defined by setting 
$$a \cdot \overline{m} = \left\{ \begin{array}{ll} \overline{a \cdot
m} & \textrm{if } a \cdot m \notin M' \\ *_{M/M'} & \textrm{ if }
a\cdot m \in M'    \end{array} \right. $$  
where $\overline{m}$ is an element $m$ in $M$ considered as an element of $M/M'$. 
If $M$ is finite, we define $|M| = \# M -1 $, i.e. the number of non-zero elements. 

In the language of Section \ref{pe_categories}, the category $\Amod$ is equivalent to the category of functors $\on{Fun}(A,\Set^{fin})$, where $A$ is viewed as a one-object category,  and is therefore proto-exact (see Example \ref{example: proto-exact} ). 

\medskip

\begin{rmk} \label{pe_structure_Amod}
It will be useful in future discussions to make explicit the proto-exact structure of $\Amod$, i.e. to have an explicit description of the objects completing the biCartesian diagrams in Definition \ref{pe_def}. We begin by noting that all morphisms in $\EE$ are of the form $X \twoheadrightarrow X/V$, where $V \subset X$ is a sub-module, and all morphisms in $\MM$ are pointed inclusions $V \hookrightarrow X$. 

\begin{itemize}
\item For the diagram appearing in Definition \ref{pe_def}(4), we may identify $j: X \twoheadrightarrow X'$ with $j: X \twoheadrightarrow X/V$. $Y'$ can then be identified with $Y/i(V)$, where all the maps are the obvious ones. 
\item For the diagram appearing in Definition \ref{pe_def}(5), $X$ can then be identified with $j'^{-1}(i'(X'))$, where again all the maps are the obvious ones. 
\end{itemize}
The admissible exact sequences in $\Amod$ are thus all isomorphic to ones of the form:
$$ V \hookrightarrow X \twoheadrightarrow X/V $$
where $V \subset X$ is a sub-module. 
\end{rmk}

The following is an immediate consequence of Remark \ref{pe_structure_Amod}

\begin{prop} \label{pe_subcat}
Let $\C$ be a full subcategory of $\Amod$ containing the $0$ module, and closed under taking sub-modules and quotients. Then $\C$ is proto-exact with the induced structure. 
\end{prop}

\medskip

$\Amod$ has the following additional properties:
\begin{enumerate}
\item $\Amod$ has a zero object, namely the one-element
pointed set $\{ * \}$. 
\item A morphism $f: (M,*_M) \rightarrow (N,*_N)$ has a kernel $(f^{-1}(*_N), *_M)$ and a cokernel $N/\on{Im}(f)$ which satisfy the usual universal properties.
\item $\Amod$ has coproducts: $M \oplus N := M \vee N := (M \sqcup N) / \langle *_M \sim *_N\rangle $ which we will call ``direct sum''. 
\item If $L \subset M \oplus N$ is an $A$--submodule, then $L= (L \cap M) \oplus (L \cap N) $.  
\item $\Amod$ has a symmetric monoidal structure  $M \otimes_A N := M \wedge_A N := M \times N / \sim $, where $\sim$ is the equivalence relation generated by  $(a \cdot m, n) \sim (m, a \cdot n)$.  
\item $M \otimes_A (N \oplus N') \simeq M \otimes_A N \oplus M \otimes_A N'$.
\end{enumerate}

\bigskip

$M$ is said to be \emph{free of rank n} if $M \simeq \oplus^n_{i=1} A$ and \emph{finitely generated} if there exists a surjection $\oplus^n_{i=1} A \twoheadrightarrow M$ of $A$--modules for some $n$. In other words, this means that there exist $m_1, \cdots, m_n \in M$ such that for every $m \in M$, there exists an $a \in A $ such that $m = a \cdot m_i$ for some $1\leq i \leq n$, and we call the $m_i$ \emph{generators}. We denote by $\Amod^{fg}$ the full subcategory of $\Amod$ consisting of finitely generated modules. 

For an element $m \in M$, define $$\textrm{Ann}_A (m) := \{ a \in A \mid a \cdot m = *_M \}.$$ This is an ideal in $A$, and  $0_A \in \textrm{Ann}_A (m) \;  \forall m \in M$. An element $m \in M$ is said to be \emph{torsion} if $\textrm{Ann}_A (m) \neq \{ 0_A \}$. The subset of all torsion elements in $M$ is an $A$--submodule, called the \emph{torsion submodule} of $M$, and denoted $M_{tor}$. An $A$--module $M$ is \emph{torsion-free} if $M_{tor} = \{ *_M \}$ and \emph{torsion} if $M_{tor} = M$. We define the ideal $\textrm{Ann}_A(M)$ by $$\textrm{Ann}_A(M) = \bigcap_{m \in M} \textrm{Ann}_A(m)  \subset A. $$ 
We note that every $M \in \Amod$ can be uniquely written $M = M_{tor} \oplus M_{tf}$, where $M_{tf}$ is torsion-free. 

\bigskip

Given a multiplicatively closed subset $S \subset A$ and an $A$--module $M$, we may form the $S^{-1} A$--module $S^{-1}M$, where
\[
S^{-1} M := \{ \frac{m}{s} \mid \; m \in M, s \in S \} 
\]
with the following equivalence relation 
$$\frac{m}{s} = \frac{m'}{s'}  \iff \exists \; s'' \in S \textrm{ such that } s's'' m = s s'' m', $$
where the $S^{-1}A$--module structure is given by $\frac{a}{s} \cdot \frac{m}{s'} := \frac{am}{ss'}.$ $Hom_A (S^{-1}M, S^{-1}N)$ has an $S^{-1}A$-module structure via 
$$ (\frac{a}{s} \cdot f) (\frac{m}{s'}) := \frac{a}{s} \cdot f(\frac{m}{s'}) $$
For $f \in A$, we define $M_f := S^{-1}_f M$, and for a prime ideal $\p \subset A$, $M_{\p} := S^{-1}_{\p} M_{\p}$. The following properties, analogous to those for modules over rings, will be useful in what follows:

\medskip

\begin{prop} \label{module_localization_prop}
Let $A$ be a monoid, and $S \subset A$ a multiplicative subset.  
\begin{enumerate}
\item Localization with respect to $S$ is an exact functor of proto-exact categories from $\Amod$ to $\sAmod$. I.e. if $ J \hookrightarrow M \twoheadrightarrow N $ is a short exact sequence in $\Amod$, then $S^{-1}J \hookrightarrow S^{-1}M \twoheadrightarrow S^{-1}N $ is a short exact sequence in $\sAmod$. 
\item Suppose $M \in \Amod$, and $J \subset M$ is a submodule. There is an isomorphism of $S^{-1}A$-modules
\[
S^{-1}(M/J) \simeq S^{-1} M / S^{-1} J
\]
\item If $M, N \in \Amod$, and $M$ is finitely generated, then there are isomorphisms of $S^{-1}A$-modules
\[
S^{-1} \on{Hom}_A (M, N) \simeq \on{Hom}_{S^{-1}A} (S^{-1} M, S^{-1} N) \simeq \on{Hom}_A (S^{-1} M, S^{-1} N)
\]
\item If $A$ is finitely generated, and $M, N \in \Amod$ are finitely generated, then $\on{Hom}_A (M,N)$ is a finitely generated $A$-module.  
\end{enumerate}

\end{prop}

\begin{proof}
The proofs of the analogous properties for commutative rings and modules carry over without any significant changes, and we comment only on $(4)$. Applying $\on{Hom}_A ( \cdot, N)$ to a surjection $\oplus^m_{i=1} A \twoheadrightarrow M$ yields  an injection $$\on{Hom}_A (M,N) \hookrightarrow \on{Hom}( \oplus^m_{i=1} A, N) \simeq \oplus^m_{i=1} \on{Hom}_A (A, N) \simeq N^{\oplus m}. $$ Since $N$ is finitely generated, so is $N^{\oplus m}$. One now checks that when $A$ is finitely generated, sub-modules of finitely generated $A$-modules are themselves finitely generated, from which the result follows. 
\end{proof}

\medskip

\subsubsection{Quasicoherent sheaves}

Let $(X, \mc{O}_X)$ be a monoidal space. A sheaf $\mc{M}$ of pointed sets on $X$ is an \emph{$\mc{O}_X$--module} if for every open set $U \subset X$, $\mc{M}(U)$ has the structure of an $\mc{O}_X (U)$--module with the usual compatibilities. Morphisms of $\mc{O}_X$-modules are defined as morphisms of sheaves of pointed sets commuting with the $\mc{O}_X$-action. In particular, given a monoid $A$ and an $A$--module $M$, there is an $\mc{O}_X$-module $\wt{M}$ on $(X=\spec A, \mc{O}_X)$, defined on basic open affines $D(f)$ by $\wt{M}(D(f)) := M_f $, whose stalk at $\p \in \spec A$ is isomorphic to $M_{\p}$. 

For a monoid scheme $(X, \mc{O}_X)$, an $\mc{O}_X$--module $\F$ is said to be \emph{quasicoherent} if for every $x \in X$  there exists an open affine $U_x \subset X$ containing $x$ and an $\mc{O}_X (U_x)$--module $M$ such that $\F \vert_{U_x} \simeq \wt{M}$. The module $\F$ is said to be \emph{coherent} if $M$ can always be taken to be finitely generated, and \emph{locally free} if $M$ can be taken to be free. For a monoid $A$, there is an equivalence of categories between the category of quasicoherent sheaves on $\spec A$ and the category of $A$--modules, given by $\Gamma(\spec A, \cdot)$. We denote by $\Qcoh(X)$ and $\Coh(X)$ the categories of quasicoherent and coherent sheaves on a monoid scheme $(X, \mc{O}_X)$. 


If $\F, \F'$ are $\mc{O}_X$-modules, then $\mc{H}om_{\mc{O}_X}(\F, \F')$ is the $\mc{O}_X$-module defined by
\[
\mc{H}om_{\mc{O}_X}(\F, \F') (U) := \on{Hom}_{\mc{O}_X \vert_U} (\F \vert_U, \F' \vert_U).
\]
It follows from Proposition \ref{module_localization_prop} that if $\F, \F' \in \Coh(X)$, then $\mc{H}om_{\mc{O}_X}(\F, \F') \in \Coh(X)$ as well. 

\begin{mydef}
Let $(X, \mc{O}_X)$ be a monoid scheme, and $\F \in \Qcoh(X)$.
\begin{itemize}
\item The module $\F$ is \emph{torsion} (resp.~\emph{torsion-free}) if the stalk $\F_{x}$ is a torsion (resp.~torsion-free) $\mc{O}_{X,x}$--module for all $x \in X$. 
\item The \emph{support} of $\F$ is the subset
\[
supp(\F) := \{ x \in X \mid \F_x \neq 0 \}.
\] 
\end{itemize}
\end{mydef}
As in the case of ordinary schemes, one easily checks the following:
\begin{prop} Let $(X, \mc{O}_X)$ be a monoid scheme. 
\begin{enumerate}
\item If  $X = \spec A$, and $M \in \Amod$ is finitely generated, then $\emph{supp}(\wt{M}) = V(\emph{Ann}_A(M))$.
\item If $\F \in \Coh(X)$, then $\emph{supp}(\F)$ is a closed subset of $X$.
\item If $\F, \F' \in \Qcoh(X)$, $\F$ is torsion, and $\F'$ is torsion-free, then $\on{\mc{H}om}_{\mc{O}_X}(\F, \F') = 0$. 
\end{enumerate}
\end{prop}

If $Z \subset X$ is a closed subset, we will denote by $\Qcoh(X)_Z, \Coh(X)_Z$ the full subcategories of sheaves $\F$ such that $\emph{supp}(\F) \subset Z$. 

Any quasicoherent sheaf of ideals $\mc{I} \subset \mc{O}_X$ determines a closed subscheme $X_{\mc{I}}$ of $X$. Its intersection with an open affine subset $U \subset X$ is isomorphic to $\spec ( \mc{O}_X (U) / \mc{I}(U)) $. We note that not all closed subschemes of $(X, \mc{O}_X)$ are of this form \cite{CHWW}.

\begin{mydef}
Let $(X, \mc{O}_X)$ be a monoid scheme, and $\F \in \Coh(X)$. The \emph{scheme-theoretic support} of $\F$ is the closed subscheme associated to the coherent sheaf of ideals $\mc{A}nn (\F) \subset \mc{O}_X$, where $\mc{A}nn(\F)(U) := \textrm{Ann}(\F(U)) \subset \mc{O}_X (U)$. 
\end{mydef}

We note that closed subset of $X$ underlying the scheme $X_{\mc{A}nn(\F)}$ is precisely $\emph{supp}(\F)$. 

\begin{mydef}
Let $\mc{I} \subset \mc{O}_X$ be a quasicoherent sheaf of ideals, and $\F \in \Qcoh(X)$. We say that $\F$ is \emph{scheme-theoretically supported} on $X_{\mc{I}} \subset X$ if $\mc{I}\cdot \F =0$. 
\end{mydef}

This implies that $\mc{I} \subset \mc{A}nn(\F)$, or that the scheme-theoretic support of $\F$ is a closed subscheme of $X_{\mc{I}}$. We denote by $\Qcoh(X)_{\mc{I}}$ the full subcategory of $\Qcoh(X)$ consisting of quasicoherent sheaves supported on $X_{\mc{I}}$.  As in the case of ordinary schemes, the functor
\begin{align*}
\iota_*: \Qcoh(X_{\mc{I}}) & \rightarrow  \Qcoh(X)_{\mc{I}} \\
\F & \rightarrow  \iota_* \F
\end{align*}
is an equivalence, which restricts to an equivalence between $\Coh(X_{\mc{I}})$ and $\Coh(X)_{\mc{I}}$.

 The bifunctors $\oplus, \otimes$ on $\Amod$ induce corresponding bifunctors on $\Qcoh(X)$:

\begin{mydef} Let $(X, \mc{O}_X)$ be a monoid scheme, and $\F, \F' \in \Qcoh(X)$.
\begin{itemize}
\item The \emph{direct sum} of $\F$ and $\F'$, denoted $\F \oplus \F'$, is the element of $\Qcoh(X)$ defined by 
\[
\F \oplus \F' (U) := \F(U) \oplus \F'(U)
\] 
\item The \emph{tensor product} of $\F$ and $\F'$, denoted $\F \otimes_{\mc{O}_X} \F'$, is the element of $\Qcoh(X)$ obtained as the sheafification of the  presheaf $U \rightarrow \F(U) \otimes_{\mc{O}_X (U)} \F'(U)$. 
\end{itemize}
\end{mydef}

\medskip

\begin{rmk} \label{stalkvsaffine}

Every monoid $A$ has a unique maximal ideal $\mathfrak{m} = A \backslash A^{\times}$, the complement of the units of $A$, where $A_{\mathfrak{m}} = A$. It follows that every affine monoid scheme has a unique closed point, and at this point the stalk of the structure sheaf can be identified with the global sections. Similarly, if $M \in \Amod$, then $\wt{M}_{\mathfrak{m}} = M = \Gamma(\mspec (A), \wt{M})$.  Furthermore, as explained in \cite{CHWW}, for every monoid scheme $X$ of finite type, the assignment $ x \in X \mapsto \mspec(\mc{O}_{X,x}) $ yields a bijection between points of $X$ and open affine subsets. It follows from this that for monoid schemes of finite type, a property of sheaves is stalk-local if and only if it is affine-local. 

\end{rmk}

\subsubsection{Gluing for quasicoherent sheaves and morphisms} \label{gluing}

We will make use of the gluing construction for quasicoherent sheaves. Namely, suppose that $\mathfrak{U} = \{ U_i \}_{i \in I}$ is an open cover of $X$, and suppose we are given for each $i \in I$ a quasicoherent sheaf $\F_i$ on $U_i$, and for each $i,j \in I$ an isomorphism $\phi_{ij}: \F_i \vert_{U_i \cap U_j} \rightarrow \F_{j} \vert_{U_i \cap U_j}$ such that 
\begin{enumerate}
\item $\phi_{ii} = id$,
\item For each $i,j,k \in I$, $\phi_{ik} = \phi_{jk} \circ \phi_{ij} $ on $U_i \cap U_j \cap U_k$.
\end{enumerate}
Then there exists a unique quasicoherent sheaf $\F$ on $X$ together with isomorphisms $\psi_i: \F \vert_{U_i} \rightarrow \F_i$, such that for each $i,j, \; \psi_j = \phi_{ij} \circ \psi_i$ on $U_i \cap U_j$. If moreover the $\F_i$ are coherent then $\F$ is coherent. If $\mathfrak{U} = \{ U_i = \spec A_i \}_{i \in I}$ is an affine open cover, the same data can be presented as a collection $\{ M_i \}$ of $A_i$-modules with the obvious compatibilities. 

Similarly, given an open cover $\mathfrak{U} = \{ U_i \}_{i \in I}$ of $X$, and quasicoherent sheaves $\F, \F'$ presented in terms of gluing data $(\F_i, \phi_{ij})$, $(\F'_i, \phi'_{ij})$, a morphism $f: \F \rightarrow \F'$ is the data of morphisms $f_i: \F_i \rightarrow \F'_i$ such that $\phi'_{ij} \circ f_i = f_j \circ \phi_{ij}$ on $U_i \cap U_j$. If $\mathfrak{U} = \{ U_i = \spec A_i \}_{i \in I}$, and $\F_i = \wt{M_i}$, $\F'_i = \wt{M'_i}$, then the $f_i$ can be taken to be $A_i$-module homomorphisms $f_i : M_i \rightarrow M'_i$. 

\begin{myeg} \label{P1sheaves}

Given the description of $\Pone$ in Example \ref{example: P1}, a quasicoherent sheaf $\F$ on $\Pone$ can be viewed as a triple $(M, M', \phi)$, where 
\begin{itemize}
\item $M$ is a $\langle t \rangle$-module. 
\item $M'$ is a $\langle t^{-1} \rangle$-module
\item $\phi: M_t \rightarrow M'_{t^{-1}}$ is an isomorphism of $\langle t, t^{-1} \rangle$-modules. 
\end{itemize}

$\F$ is coherent if $M, M'$ are finitely generated. Similarly, a quasicoherent subsheaf $\F' \subset \F$ is given by submodules $N \subset M$, $N' \subset M'$, such that $\phi$ restricts to an isomorphism $N_t \rightarrow N_{t^{-1}}$. We have for instance the following:

\begin{itemize}
\item $(\langle t \rangle, \langle t^{-1} \rangle, \phi_n)$ where $\phi_n (u) = t^{-n} u$, $n \in \mathbb{Z}$. We denote this coherent sheaf by $\mc{O}(n)$ in analogy with usual setting. $\mc{O}(n)$ is locally free of rank one. 
\item $( \langle t \rangle/ (t^m), 0, 0)$, and $(0, \langle t^{-1} \rangle/ (t^{-m}),0)$, $m \geq 0$. We call these sheaves $\mc{T}_{0, m}$ and $\mc{T}_{\infty, m}$ respectively. These are torsion sheaves supported at $0, \infty$ respectively.  
\end{itemize}

If $0 \leq k, r \leq n$, then $(t^k\cdot M, t^{-r} \cdot M', \phi_n)$ is a  coherent subsheaf of $\mc{O}(n)$ isomorphic to $\mc{O}(n-k-r)$, which yields a short exact sequence (see the following section for the proto-exact structure on $\Qcoh(\Pone)$)
\[
\mc{O}(n-k-r) \hookrightarrow \mc{O}(n) \twoheadrightarrow \mc{T}_{0,k} \oplus \mc{T}_{\infty, r}
\]

\end{myeg}

\subsubsection{Proto-exact structure on $\Qcoh(X)$ and $\Coh(X)$}

We proceed to describe the proto-exact structure on the categories $\Qcoh(X)$ and $\Coh(X)$, which is obtained from the corresponding proto-exact structure on $\Amod$ explained in Example \ref{example: proto-exact} and Remark \ref{pe_structure_Amod}. Let $(X, \mc{O}_X)$ be a monoid scheme, and $\{ U_i = \mspec(A_i) \}_{i \in I}$ an affine cover.  For $\F, \F' \in \Qcoh(X)$, define 
$$\MM(\F,\F') := \{ \phi \in \on{Hom}_{\mc{O}_X} (\F, \F') \mid \phi \vert_{U_i} \in \MM(\F \vert_{U_i}, \F' \vert_{U_i}) ~\forall i \in I \}, $$
$$\EE(\F,\F') := \{ \phi \in \on{Hom}_{\mc{O}_X} (\F, \F') \mid \phi \vert_{U_i} \in \EE(\F \vert_{U_i}, \F' \vert_{U_i}) ~\forall i \in I \}, $$
where for every $i \in I$, $\phi \vert_{U_I}$ denotes the restriction of the morphism $\phi$ to $U_i = \mspec{A_i}$, where it can be identified with a morphism in $A_i - \on{Mod}$. In other words, $\phi$ is an admissible mono/epi if it locally admissible as a map of $A_i$-modules. 

By Remark \ref{stalkvsaffine}, this is equivalent to requiring the induced morphism $\phi_x \in \on{Hom}(\F_x, \F'_x) $ on stalks to be an admissible mono/epi for every $x \in X$, and furthermore independent of the affine cover of $X$. 

\begin{prop} \label{coh_pe_thm}
Let $(X, \mc{O}_X)$ be a monoid scheme.
\begin{enumerate}
\item $(\Qcoh(X), \MM, \EE)$ and $(\Coh(X), \MM, \EE)$ are proto-exact.
\item If $X = \spec A$, then the functor $$ \wt{}: \Amod \mapsto \Qcoh(X) $$ is an exact equivalence, restricting to an exact equivalence $$ \wt{}: \Amod^{fg} \mapsto \Coh(X).$$
\item If $Z \subset X$ is a closed subset, then $(\Qcoh(X)_Z, \MM, \EE)$ and $(\Coh(X)_Z, \MM, \EE)$ are proto-exact.
\item If $\mc{I} \subset \mc{O}_X$ is a quasicoherent sheaf of ideals, then $(\Qcoh(X)_{\mc{I}}, \MM, \EE)$ and $(\Coh(X)_{\mc{I}}, \MM, \EE)$ are proto-exact subcategories of $\Qcoh(X), \Coh(X)$ resp. 
\end{enumerate}
\end{prop}

\begin{proof}
For (1), we must verify that with the above definition, the bicartesian squares in Definition \ref{pe_def} involving $\MM$ and $\EE$ exist in $\Qcoh(X)$ (resp. $\Coh(X)$).   If $X = \spec A$ is affine, $\Qcoh(X)$ is equivalent to $\Amod$, and the statement follows from the fact that $\Amod$ is proto-exact. For general $X$, we may pass to an affine cover $\{ U_{i} = \spec A_{i} \} $, and compute the push-outs/pull-backs in Definition \ref{pe_def} in $\on{Mod}_{\on{A_{i}}}$. It follows from Remark \ref{pe_structure_Amod} and part (1) of Proposition \ref{module_localization_prop} that these glue. The proof for $\Coh(X)$ is identical, working with finitely generated modules. 

Part (2) is immediate from the definition. 

Part (3) follows from the fact that if $$ \ses{\F}{\F'}{\F''} $$ is an admissible short exact sequence, then $supp(\F') = supp(\F) \cup supp(\F'') .$

For part (4) $\Qcoh(X)_{\mc{I}}$ (resp.~$\Coh(X)_{\mc{I}}$) are equivalent to $\Qcoh(X_{\mc{I}})$ (resp.~$\Coh(X_{\mc{I}})$) via $\iota_*$ as explained above, and the latter are manifestly proto-exact. Alternatively, one observes that the question is local, and can be verified on affines, where the statement follows Proposition \ref{pe_subcat}. 

\end{proof}

\subsection{Realizations of monoid schemes and quasicoherent sheaves} \label{realization}
\medskip

In this section, we recall realizations of monoid schemes and quasicoherent sheaves. We refer the reader to \cite{CHWW} for details. 

Let $k$ be a commutative ring, and $A$ a monoid. Let $k[A]$ be the monoid algebra:
\[
k[A] := \left\{ \sum r_i a_i \mid a_i \in A, a_i \neq 0, r_i \in k \right\}
\]
with multiplication induced from the monoid multiplication.\footnote{Strictly speaking, we identify $0_k$ and $0_A$.} It is clear that any morphism $f:A_1 \to A_2$ of monoids induces a ring homomorphism $f_k:k[A_1] \to k[A_2]$ and hence we have a functor
\[
-\otimes k:\textbf{Monoids} \longrightarrow \textbf{Rings}, \quad 
A \otimes k:=k[A].
\]
which we refer to as \emph{scalar extension to $k$}. 
The functor $- \otimes k$ has a right adjoint sending $k$ to $(k, \times)$ - its underlying monoid, and so preserves colimits. 

Scalar extension may be used to define a functor  
\begin{align*}
-_k : \Msch  & \mapsto \on{Sch} / \on{Spec} k \\
               X & \mapsto X_k
\end{align*}
defined on affine monoid schemes by $\spec A_k := \on{Spec}(k[A])$, and for a general $X$ by gluing these over an affine cover. Following \cite{CHWW}, we refer to $X_k$ as the \emph{$k$-realization of $X$}. It may be thought of as a base-change from $\spec \fun$ to $\on{Spec} k $. The realization functor preserves limits.

One can similarly construct the scalar extension for quasicoherent sheaves. Given an $A$--module $M$, let
\[
k[M] := \left\{ \sum r_i m_i \mid m_i \in M, m_i \neq *, r_i \in k \right\}.
\]
$k[M]$ naturally inherits the structure of an $k[A]$--module. We may use this to define a realization functor
\begin{align*}
-_k: \Qcoh(X) & \mapsto \Qcoh(X_k) \\
\mc{F} & \mapsto \F_k
\end{align*}
It is defined on affines by assigning to $\wt{M}$ on $\mspec{A}$ the quasicoherent sheaf
\[
\wt{M}_k := \wt{k[M]}
\]
on $(\mspec(A))_k = \on{Spec}(k[A])$, and for a general monoid scheme by gluing in the obvious way. We have that $\F_k \in \Coh(X_k)$ when $\F \in \Coh(X)$. We have the following:

\begin{prop}
Let $X$ be a monoid scheme, and $k$ a commutative ring. $\F \mapsto \F_k$ is an exact functor from $\Qcoh(X)$ to $\Qcoh(X_k)$. 
\end{prop}

\begin{proof}
Since exactness can be checked on stalks, it suffices to verify this when $X = \spec A$ is affine. It is clear that if $$ \ses{M}{P}{N} $$ is an admissible exact sequence in $\Amod$, then $$ 0 \rightarrow k[M] \rightarrow k[P] \rightarrow k[N] \rightarrow 0 $$ is an exact sequence of $k[A]$-modules. 
\end{proof}

Given a monoid scheme $X$ and $\F \in \Qcoh(X)$, we have for each open $U \subset X$ a map
\begin{equation} \label{basechange_sec}
\phi_k (U):  k[\Gamma(\F, U)] \rightarrow \Gamma(\F_k, U_k) 
\end{equation}
defined as the unique $k$--linear map with the property that $\phi_k (U) (s) = s \;  \forall s \in \Gamma(\F,U)$. When $U$ is understood, we will refer to this map simply as $\phi_k$.  

\subsection{Fans, monoid schemes, and toric varieties} \label{fans}

In this section we recall some basics of the relationship between fans, monoid schemes, and toric varieties. We refer the interested reader to \cite{CHWW} for details. 

We begin by recalling some terminology pertaining to cones and fans. If $N$ is a free abelian group of finite rank, we will denote by $N_{\mathbb{R}}$ the vector space $N \otimes_{\mathbb{Z}} \mathbb{R}$. Recall that: 
\begin{itemize}
\item A \emph{rational polyhedral cone} $\sigma \subset N_{\mathbb{R}}$ is a cone generated by finitely many elements $u_1, u_2, \cdots, u_s \in N$:
 \[ 
 \sigma = \{ \lambda_1 u_1 + \cdots \lambda_s u_s \mid \lambda_i \geq 0 \}. 
 \]
\item $\sigma$ is \emph{strongly convex} if $\sigma \cap (-\sigma) =\{0\}$.
\item The \emph{dimension} of $\sigma$ is the dimension of the smallest subspace of $N_\mathbb{R}$ containing $\sigma$. 
\item A \emph{face} of $\sigma$ is the intersection $\{ l=0 \} \cap \sigma$, where $ l \in N^*_{\mathbb{R}}$ is a linear form which is non-negative on $\sigma$. 
\end{itemize}

Let now $M = \on{Hom}_{\mathbb{Z}}(N, \mathbb{Z})$, and denote by $\langle, \rangle : M \times N \rightarrow \mathbb{Z}$ the canonical pairing. We have $M_{\mathbb{R}} \simeq N^*_{\mathbb{R}}$.  If $\sigma \subset N_{\mathbb{R}}$ is a strongly convex rational polyhedral cone, the \emph{dual cone} $\sigma^{\vee} \subset M_{\mathbb{R}}$ is 
\[
\sigma^{\vee} :=\{ m \in M_{\mathbb{R}} \mid \langle m, u \rangle \geq 0 ~\forall u \in \sigma \}.
\]
By Gordan's lemma, $S_{\sigma} :=\sigma^{\vee} \cap M$ is a finitely generated monoid, from which we construct the affine monoid scheme $X_{\sigma} := \spec S_{\sigma}$. The realization $X_{\sigma, k} = \Spec k[S_{\sigma}]$ to a field $k$ has the structure of an affine toric variety. The torus $T$ corresponds to $M$ (or $N$), and can be written $T=\spec M$, with $T_k =\Spec k[M]$. In algebraic terms, the torus action
\[
\Spec k[M] \times_k \Spec k[S_{\sigma}] \rightarrow \Spec k[S_{\sigma}]
\]
corresponds to the co-multiplication
\[
S_{\sigma} \to S_\sigma \otimes M, \quad t^m \rightarrow t^m\otimes t^m
\]
where we have written elements of $S_{\sigma} \subset M$ multiplicatively as $t^m, m \in M$, with multiplication given by $t^m \cdot t^{m'} = t^{m+m'}$. 

We may create more general monoid schemes by gluing the $X_{\sigma}$ above from cones assembled into a \emph{fan}.   A \emph{fan} consists of a finite collection $\Delta$ of cones in $N \otimes_{\mathbb{Z}} \mathbb{R}$ satisfying the following properties:
\begin{itemize}
\item Each $\sigma \in \Delta$ is a strongly convex rational polyhedral cone.
\item If $\sigma \in \Delta$, and $\tau$ is a face of $\sigma$, then $\tau \in \Delta$.
\item If $\sigma, \sigma' \in \Delta$, then $\sigma \cap \sigma'$ is a face of each of $\sigma, \sigma'$. 
\end{itemize}

If $\sigma \in \Delta$, and $\tau$ is a face of $\sigma$, then $S_{\sigma} \subset S_{\tau}$ is a submonoid, inducing an open embedding $X_{\tau} \subset X_{\sigma}$. We may now construct a monoid scheme $X_{\Delta}$ by gluing $X_{\sigma_1}$ and $X_{\sigma_2}$ along $X_{\sigma_1 \cap \sigma_2}$ $\forall \sigma_1, \sigma_2 \in \Delta$. The toric variety $X_{\Delta, k}$ over $\Spec k$ is obtained by gluing the affine toric varieties $X_{\sigma, k}, X_{\tau, k}$ along $X_{\sigma \cap \tau, k}$ $\forall \sigma, \tau \in \Delta$. 

\begin{myeg}
The projective line $\mathbb{P}^1_k$, as a toric variety, arises from the following fan:
\[
\begin{tikzpicture}[x=.4cm,y=.4cm]

\draw[->, line width=0.8mm] (0,0) -- (8,0) node[above] {};
\draw[<-,line width=0.8mm] (-8,0) -- (0,0) node[above] {};

\draw[thick, fill=black] (0,0) circle(0.25);

grid[xstep=.5cm, ystep=.5cm] (4,4);
\draw (0,-1.5) node[below] {The fan $\Delta$ for $\mathbb{P}^1_k$};

\draw (4,2) node[below] {$\sigma_0$};

\draw (-4,2) node[below] {$\sigma_1$};
\end{tikzpicture}
\]
From cones in $\Delta$, one obtains the following affine monoid schemes (c.f. Example \ref{example: P1}):
\begin{enumerate}
	\item 
$X_{\sigma_0}=\spec S_{\sigma_0}=\spec (\{x^n \mid n \in \mathbb{N}\}) = \mathbb{A}^1=\{\{0\}, \langle x \rangle\}$, 
\item 
$X_{\sigma_1}=\spec S_{\sigma_1}=\spec (\{x^{-n} \mid n \in \mathbb{N}\})= \mathbb{A}^1=\{\{0\}, \langle x^{-1} \rangle\}$, 
\item 
$X_{\sigma_0 \cap \sigma_1}=\spec S_{\sigma_0 \cap \sigma_1}=\spec (\{x^n \mid n \in \mathbb{Z}\})= \mathbb{T}=\{\{0\}\}$.
\end{enumerate}
Taking $k$-realizations, one obtains the affine line $\mathbb{A}^1_k=X_{\sigma_0, k} = X_{\sigma_1, k} $ and the torus $T_k=X_{\sigma_0 \cap \sigma_1, k}$. By gluing $X_{\sigma_0}$ and $X_{\sigma_1}$ along $X_{\sigma_0 \cap \sigma_1}$, we obtain the monoid projective line $\mathbb{P}^1$, whose realization is $\mathbb{P}^1_k$. 
\end{myeg}

\begin{myeg} \label{P2}
The projective plane $\mathbb{P}^2_k$, as a toric variety, arises from the following fan:
	\[
\begin{tikzpicture}[x=.6cm,y=.6cm]
	\draw[style=dashed, fill=red!20] (0,0) -- (0,4) -- (4,4) -- (4,0)-- cycle;

	\draw (2,2) node[below] {$\sigma_0$};	
	
	\draw[style=dashed, fill=green!20] (0,0) -- (0,4) -- (-4,4) -- (-4,-4) -- cycle;
	
	\draw (-2,1) node[below] {$\sigma_1$};	

	\draw[, style=dashed, fill=gray!40] (0,0) -- (-4,-4) -- (4,-4) -- (4,0) 
	-- cycle;
	
	\draw (1,-1.5) node[below] {$\sigma_2$};
	
\draw[line width=0.8mm] (0,0) -- (0,4.2) node[above] {};
\draw[line width=0.8mm] (0,0) -- (4.2,0) node[above] {};
\draw[line width=0.8mm] (0,0) -- (-4.2,-4.2) node[above] {};

grid[xstep=.5cm, ystep=.5cm] (4,4);
	\draw (0,-4.75) node[below] {The fan $\Delta$ for $\mathbb{P}^2$};
	\end{tikzpicture}\]
From $\sigma_i$, one obtains the following three affine monoid schemes:
\begin{enumerate}
	\item 
	$X_{\sigma_0}=\spec (\langle x_2, x_1 \rangle) = \mathbb{A}^2$ (the monoid affine plane), 
	\item 
	$X_{\sigma_1}=\spec (\langle x_1^{-1}, x_1^{-1}x_2 \rangle) = \mathbb{A}^2$, 
	\item 
$X_{\sigma_2}=\spec (\langle x_1x_2^{-1}, x_2^{-1}\rangle) = \mathbb{A}^2$. 
\end{enumerate}	

For each $i\neq j$, $\tau_{ij}=\sigma_i \cap \sigma_j$ provides the gluing data, and one obtains the monoid projective plane $X_\Delta = \mathbb{P}^2$. In particular, $X_{\Delta, k} =\mathbb{P}^2_k$.
\end{myeg}

A cone $\sigma \subset N_{\mathbb{R}}$ is \emph{smooth} if it is generated by a subset of a basis of $N$. The toric variety $X_{\Delta, k}$ is smooth $\iff$ every $\sigma \in \Delta$ is smooth. If $\sigma \in \Delta$ is smooth of maximal dimension $n=dim(N_{\mathbb{R}}$, generated by $u_1, \cdots, u_n$, then $\sigma^{\vee}$ is also smooth, generated by $u^{\vee}_1, \cdots, u^{\vee}_n$ where $\{ u^{\vee}_i$ is the dual basis. This implies that $\mspec(S_{\sigma}) \simeq \mathbb{A}^n$. Thus, if $\Delta$ consists of smooth cones of maximal dimension (such as for instance the fan of $\mathbb{P}^2$ above), $X_{\Delta}$ has a cover by affines isomorphic to $\mathbb{A}^n$. Finally, $X_{\Delta, k}$ is projective $\iff$ $\Delta$ is the \emph{normal fan} of a $rk(N)$-dimensional lattice polytope in $M_{\mathbb{R}}$ (see \cite{coxtoric} for an explanation of normal fan). 

In \cite{CHWW}, Corti\~n{a}s, Haesemeyer, Walker, and Weibel characterize which monoid schemes arise from fans. To be precise, they show that for a fan $\Delta$, the monoid scheme $X_\Delta$ is a separated, connected, torsion-free, normal monoid scheme of finite type. Also, conversely, if $X$ is such a monoid scheme (toric monoid scheme), they show that one can construct a fan $\Delta$ so that $X=X_\Delta$. For details, see \cite[Section 4]{CHWW}.

\subsection{Decomposing modules and coherent sheaves}

In this section we show that coherent sheaves on a monoid scheme admit finite canonical decompositions into indecomposables, showing that a version of the Krull-Schmidt theorem holds in this context.  

\begin{mydef}
Let $A$ be a monoid. An $A$-module $M$ is \emph{indecomposable} if it cannot be written $M = M' \oplus M''$ for non-zero $A$-modules $M', M''$. Similarly, a quasicoherent sheaf $\F$ on a monoid scheme $X$ is \emph{indecomposable} if it cannot be written $\F = \F' \oplus \F''$ for non-zero quasicoherent sheaves $\F', \F''$. 
\end{mydef}

Thus, a quasicoherent sheaf $\F = \wt{M}$ on $X = \spec A$ is indecomposable if and only if $M$ is indecomposable as an $A$-module. 
It follows from property $(4)$ of the category $\Amod$ (see the paragraph after Proposition \ref{pe_subcat}) that if $\F,\F'$ are quasicoherent $\mc{O}_X$--modules, and $\mc{G} \subset \F\oplus \F'$ is a quasicoherent subsheaf, then \mbox{$\mc{G} = (\mc{G}\cap \F) \oplus (\mc{G} \cap \F' )$}, where for an open subset $U \subset X$,  
\begin{equation*} (\mc{G} \cap \F) (U) := \mc{G}(U) \cap \F (U). \end{equation*} We thus obtain:

\begin{prop} \label{subobjects} 
Suppose that $\F, \F'$ are non-zero quasicoherent sheaves on a monoid scheme $(X, \mc{O}_X)$, and that $\mc{G}$ is an indecomposable quasicoherent subsheaf of $\F \oplus \F'$
Then $\mc{G} \subset \F$ or $\mc{G} \subset \F'$. 
\end{prop}

\begin{prop}\label{proposition: decomposition}
\begin{enumerate}
\item Let $A$ be a monoid, and $M \in \Amod$ a finitely generated $A$-module. Then $M$ can be written as a finite direct sum
\[
M \simeq M_1 \oplus M_2  \oplus \cdots \oplus M_k
\]
where $M_i, i=1 \cdots k$ is indecomposable. Moreover, if $M \simeq M'_1 \oplus M'_2 \oplus \cdots \oplus M'_{k'}$ is another such decomposition into indecomposable modules, then $k=k'$, and there exists a permutation $\sigma \in S_k$ such that $M_i \simeq M'_{\sigma(i)}$ for $i=1, \dots, k$. 

\item Let $(X, \mc{O}_X)$ be a quasi-compact monoid scheme. Then every $\F \in \Coh(X)$ can be written as a finite direct sum
\begin{equation} \label{Fdecomp}
\F \simeq \F_1 \oplus \F_2 \oplus \cdots \oplus \F_m
\end{equation}
where each $\F_i \in \Coh(X)$ is indecomposable. Moreover, if $ \F \simeq \F'_1 \oplus \F'_2 \oplus \cdots \oplus \F'_{m'}$ is another such decomposition into indecomposable coherent sheaves, then $m=m'$, and there exists a permutation $\sigma \in S_m$ such that $\F_i \simeq \F'_{\sigma(i)}$.

\end{enumerate}
\end{prop}

\begin{proof}
For the first part, consider the equivalence relation on $M \backslash 0$ generated by $m \sim am$, for $a \in A$, $m \in M$. The decomposition into indecomposable factors is easily seen to coincide with equivalence classes for this relation, by adjoining $0$ to each. To see that a finite number of indecomposable factors occur, we note that each must contain at least one generator, and so the number of factors is bounded above by the minimum number of generators for $M$. 

For the second part, we first show that a finite decomposition into indecomposable sheaves exists. Let $\mathcal{U}=\{U_1, \cdots, U_r\}$ be a finite affine cover of $X$ (which exists since $X$ is quasi-compact). For each $i$ we have $\F \vert_{U_i} \simeq \wt{N}_i$ for some finitely generated $\mc{O}_X (U_i)$ - module $N_i$. By the first part of the proposition, we can decompose $N_i$ uniquely into finitely many indecomposable factors $N_i \simeq N_{i, 1} \oplus \cdots \oplus N_{i, k_i}$. Suppose now that $\F \simeq \F_1 \oplus \cdots \oplus \F_s$ is some decomposition of the coherent sheaf $\F$ into not necessarily indecomposable non-zero summands. Each $\wt{N}_{i,j}$ belongs to exactly one of the summands, yielding a map from the set $\{ \wt{N}_{i,j} \}$ to the set $\{ \F_k \}$. Moreover, since each $\F_k$ must be non-zero on at least one of the affines $U_1, \cdots, U_r$, we see this map is surjective. It follows that a finite decomposition of $\F$ into indecomposable summands exists. 

We now prove uniqueness. Suppose that 
\[
\F \simeq \F_1 \oplus \F_2 \oplus \cdots \oplus \F_m \simeq  \F'_1 \oplus \F'_2 \oplus \cdots \oplus \F'_{m'}
\]
where $\F_i, \F'_j$ are indecomposable. Denote by $\iota_i : \F_i \hookrightarrow \F$ (resp.  $\iota'_j: \F'_j \hookrightarrow \F$) the inclusions.  By Proposition \ref{subobjects}, $\F_1 \subset \F'_r$ for some $1\leq r \leq m'$. The composition
\[
\F_1 \hookrightarrow \F \twoheadrightarrow \F'_r \hookrightarrow \F  \twoheadrightarrow \F_1
\]
is an isomorphism, implying that $\F'_r  \hookrightarrow \F  \twoheadrightarrow \F_1$ is a surjection. By Proposition \ref{subobjects} again, $\F'_r \subset \F_s$ for a unique $1 \leq s \leq m$, and since $\F'_r \cap \F_1 \neq 0$, we conclude that $s=1$.  Thus $\F'_r  \hookrightarrow \F  \twoheadrightarrow \F_1$ is an injection, and therefore an isomorphism. The claim now follows by induction.

\end{proof}

The following simple proposition follows immediately from the fact that an isomorphism of pointed sets is a permutation:

\begin{prop} \label{splitting_isomorphisms}
\begin{enumerate}
\item Let $A$ be a monoid, and $M,M' \in \Amod$, such that $M = N \oplus K$, $M'=N' \oplus K'$. Suppose that $\phi: M \rightarrow M'$ is an isomorphism that restricts to an isomorphism $\phi_1 : N \rightarrow N'$. Then $\phi = \phi_1 \oplus \phi_2$, where $\phi_2 = \phi \vert_K : K \rightarrow K'$.
\item Let $(X, \mc{O}_X)$ be a monoid scheme, and $\F, \F' \in \Coh(X)$ such that $\F= \G \oplus \mc{H}$, $\F'=\G' \oplus \mc{H}'$, and $\psi: \F \rightarrow \F'$ an isomorphism restricting to an isomorphism $\psi_1: \G \rightarrow \G'$. Then $\psi = \psi_1 \oplus \psi_2 $, where $\psi_2 = \psi \vert_{\mc{H}}: \mc{H} \rightarrow \mc{H}'$.
\end{enumerate}
\end{prop}


The following elementary result will be used later to classify indecomposable sheaves:

\begin{prop} \label{splitting_sheaves}
Let $X$ be a monoid scheme, and $U, V \subset X$ open subsets such that $X = U \cup V$. Suppose that $\F$ is a coherent sheaf on $X$ such that
\begin{itemize}
\item $\F \vert_{U} \simeq \F' \oplus \F''$, where $\F', \F''$ are non-zero.
\item $\F' \vert_{U \cap V} = 0$.
\end{itemize}
Then $\F$ is decomposable. 
\end{prop}

\begin{proof}
Let $\G = \F \vert_{V}.$ By Section \ref{gluing} we may view $\F$ as glued from $\F' \oplus \F''$ to $\G$ via an isomorphism $\phi: \F' \oplus \F'' \vert_{U \cap V} \rightarrow \G \vert_{U \cap V} $. Since $\F' \vert_{U \cap V} =0$, $\phi$ is an isomorphism $\F'' \vert_{U \cap V} \rightarrow \G \vert_{U \cap V}$. Thus, $\F = \mc{H} \oplus \mc{H}'$, where $\mc{H} = (\F'',U) \simeq_{\phi} (\G, V)$, and $\mc{H}' = (\F',U)\simeq_0 (0, V)$. 
\end{proof}

\section{ Type-$\alpha$ sheaves and Hall Algebras} \label{section: Type sheaves and Hall Algebras} 

Having reviewed the basics of monoid schemes and coherent sheaves on them, we now turn to Hall algebras. If $X$ is an ordinary projective variety over a finite field, the category $\Coh(X)$ is finitary (see \cite{S}), and so one may define and study its Hall algebra. In the world of monoid schemes, the situation is more complicated. If $X$ is a monoid scheme, the category $\Coh(X)$ is no longer finitary in general, as seen in \cite[Example 4]{lorscheid2018quasicoherent} even when $X=\mathbb{P}^1$.  There is however a class of sheaves we call \emph{type}-$\alpha$ sheaves, which are well-behaved in this regard, and can be used to define a Hall algebra. 

\subsection{Type-$\alpha$ modules} \label{subsection: type=alpha modules}

Recall from part $(2)$ of Example \ref{example: proto-exact} that $\vfun$ is the proto-abelian sub-category of $\Set^{fin}$ with objects finite pointed sets and morphisms satisfying the condition \eqref{spec_condition}. 

\begin{mydef}(\cite[Definition 2.2.1]{Sz0}) \label{typeamod}
Let $A$ be a monoid. $M \in \Amod$ is said to be of \typea if it is isomorphic to a module in $\on{Fun}(A,\vfun)$, where $A$ is viewed as a one-object category.
\end{mydef}

In other words, $M$ is of \typea if for $a \in A$, $x,y \in M$, 
	\begin{equation} \label{cond_alpha}
	ax=ay \iff x=y \textrm{ OR } ax=ay=0.
	\end{equation}
Equivalently, the endomorphism of $M$ sending $x$ to $ax$ satisfies the condition \eqref{spec_condition} for all $a \in A, x \in M$. We denote by $\Amoda$ the (non-full) proto-abelian subcategory of $\Amod$ consisting of \typea modules with morphisms satisfying the condition \eqref{spec_condition}. The proto-abelian property of $\Amoda$ implies in particular that all morphisms have kernels and cokernels. We note that if $A$ is partially cancellative, then $\Amoda$ contains all ideals $I$, as well as all quotients of ideals $I/J$ for $J \subset I \subset A$. 

\begin{prop} \label{alpha_localization}
Let $A$ be a monoid, and $S \subset A$ a multiplicative subset.
\begin{enumerate}
\item If $M \in \Amoda$, then $S^{-1}M$ is a \typea \; $S^{-1}A$-module. 
\item Localization with respect to $S$ is an exact functor of proto-abelian categories from $\Amoda$ to $\sAmod^{\alpha}$. I.e. if $ J \hookrightarrow M \twoheadrightarrow N $ is a short exact sequence in $\Amoda$, then $S^{-1}J \hookrightarrow S^{-1}M \twoheadrightarrow S^{-1}N $ is a short exact sequence in $\sAmod^{\alpha}$. 
 \end{enumerate}
\end{prop}

\begin{proof}
Suppose $M \in \Amoda$, and $ \frac{a}{s} \in S^{-1}A, \frac{m}{t}, \frac{m'}{t'} \in S^{-1}M$, such that
\[
\frac{a}{s} \cdot   \frac{m}{t} =  \frac{a}{s}\cdot  \frac{m'}{t'} 
\] 
It follows that there is $s' \in S$ such that $s' s t' a m = s' s ta m' \in M$, or $(s's a ) t' m = (s's a) t m'$. By \eqref{cond_alpha}, either $t'm = tm'$, implying $ \frac{m}{t} =  \frac{m'}{t'} \in S^{-1}M$, or both sides are $0$. In the latter case we have $(s'st') a m =0$, and since $s'st' \in S$, $\frac{a}{s} \cdot \frac{m}{t} =0$, and by the same reasoning $\frac{a}{s} \cdot \frac{m'}{t'} =0$. The second part follows from Proposition \ref{module_localization_prop}. 
\end{proof}

\begin{myeg} \label{alphaA1}
Let $A = \langle t \rangle$. Given $M \in \Amod$, we may construct a directed graph $\Gamma_M$ which completely describes the isomorphism class of $M$ as follows: the vertices of $\Gamma_M$ are the non-zero elements of $M$, with directed edges from $m$ to $t \cdot m$ for every non-zero $m$. With this, one may observe that every vertex of $\Gamma_M$ has at most one outgoing edge. Recall that for a directed graph, a vertex is said to be a \emph{leaf} (resp.~\emph{root}) if it has no incoming (resp.~outgoing) edges. Then, one can easily see that elements of $M$ corresponding to leaves of $\Gamma_M$ form a minimal set of generators for $M$ as an $A$--module. If $M, N$ are $A$--modules, then one has $\Gamma_{M\oplus N} = \Gamma_M \sqcup \Gamma_N$ - i.e. direct sums of $A$--modules (or equivalently coherent sheaves on $\mathbb{A}^1$) correspond to disjoint unions of associated directed graphs.

It is well-known (see, for instance, \cite{Sz2}) that the connected components (i.e. those corresponding to indecomposable modules) of $\Gamma_M$ can be of three distinct types:

\begin{enumerate}
\item a rooted tree - i.e. the underlying undirected graph of $\Gamma$ is a tree possessing a unique root, such that there is a unique directed path from every vertex to the root (see {\bf Figure 1}). 
\item a graph obtained by joining a rooted tree to the initial vertex of $\Gamma_{\langle t \rangle}$ (see {\bf Figure 2}). 
\item a graph obtained by gluing rooted trees (see {\bf Figure 3}) to an oriented cycle. 
\end{enumerate}

\begin{minipage}{.5\textwidth}
\begin{center}
\begin{tikzpicture}
\draw [ultra thick,->] (0,0) -- (0.9,0.9);
\draw [fill] (0,0) circle [radius=0.1];
\draw [fill] (1,1) circle [radius=0.1];
\draw [ultra thick,->] (2,0) -- (1.1,0.9);
\draw [fill] (2,0) circle [radius=0.1];
\draw [ultra thick,->] (1,1) -- (1.9,1.9);
\draw [fill] (2,2) circle [radius=0.1];
\draw [fill] (4,0) circle [radius=0.1];
\draw [fill] (3,1) circle [radius=0.1];
\draw [ultra thick,->] (4,0) -- (3.1, 0.9);
\draw [ultra thick,->] (3,1) -- (2.1,1.9);
\draw [fill] (2,1) circle [radius=0.1];
\draw [ultra thick,->] (2,1) -- (2,1.9);
\draw [fill] (2,3) circle [radius=0.1];
\draw [ultra thick,->] (2,2) -- (2,2.9);
\draw [fill] (3,2) circle [radius=0.1];
\draw [ultra thick,->] (3,2) -- (2.1,2.9);
\draw [ultra thick,->] (2,3) -- (2,3.9);
\draw [fill] (2,4) circle [radius=0.1];
\node at (2,-1) {Figure 1};
\end{tikzpicture}
\end{center}
\end{minipage}
\begin{minipage}{.5\textwidth}
\begin{center}
\begin{tikzpicture}
\draw [ultra thick,->] (0,0) -- (0.9,0.9);
\draw [fill] (0,0) circle [radius=0.1];
\draw [fill] (1,1) circle [radius=0.1];
\draw [ultra thick,->] (2,0) -- (1.1,0.9);
\draw [fill] (2,0) circle [radius=0.1];
\draw [ultra thick,->] (1,1) -- (1.9,1.9);
\draw [fill] (2,2) circle [radius=0.1];
\draw [fill] (4,0) circle [radius=0.1];
\draw [fill] (3,1) circle [radius=0.1];
\draw [ultra thick,->] (4,0) -- (3.1, 0.9);
\draw [ultra thick,->] (3,1) -- (2.1,1.9);
\draw [ultra thick,->] (2,2) -- (2,2.9);
\draw [fill] (2,3) circle [radius=0.1];
\draw [ultra thick,dotted,->] (2,3) -- (2,4);
\node at (2,-1) {Figure 2};
\end{tikzpicture}
\end{center}
\end{minipage}

\begin{center}
\begin{tikzpicture}
\draw [ultra thick,->] (0,0) -- (0.9,0);
\draw [ultra thick,->] (1,0) -- (1,0.9);
\draw [ultra thick,->] (1,1) -- (0.1,1);
\draw [ultra thick,->] (0,1) -- (0,0.1);
\draw [ultra thick,->] (0,-1) -- (0,-0.1);
\draw [ultra thick,->] (-1,-1) -- (-0.1,-0.1);
\draw [ultra thick,->] (2,2) -- (1.1,1.1);
\draw [ultra thick,->] (2,3) -- (2,2.1);
\draw [ultra thick,->] (3,2) -- (2.1,2);
\draw [fill] (0,0) circle [radius=0.1];
\draw [fill] (1,1) circle [radius=0.1];
\draw [fill] (1,0) circle [radius=0.1];
\draw [fill] (0,1) circle [radius=0.1];
\draw [fill] (-1,-1) circle [radius=0.1];
\draw [fill] (0,-1) circle [radius=0.1];
\draw [fill] (2,2) circle [radius=0.1];
\draw [fill] (2,3) circle [radius=0.1];
\draw [fill] (3,2) circle [radius=0.1];
\node at (1,-2) {Figure 3};
\end{tikzpicture}
\end{center}

The modules of \typea are then easily seen to be those where every vertex has either $0$ or $1$ incoming edge, either ladders (finite or infinite)  (see {\bf Figure 4}) or pure cycles (i.e. those without trees attached - see {\bf Figure 5}). 

\begin{minipage}{.5\textwidth}
\begin{center}
\begin{tikzpicture}
\draw [ultra thick,->] (2,0) -- (2,0.9);
\draw [ultra thick,->] (2,1) -- (2,1.9);
\draw [fill] (2,0) circle [radius=0.1];
\draw [fill] (2,1) circle [radius=0.1];
\draw [ultra thick,->] (2,2) -- (2,2.9);
\draw [fill] (2,3) circle [radius=0.1];
\draw [ultra thick,dotted,->] (2,3) -- (2,4);
\node at (2,-1) {Figure 4};
\end{tikzpicture}

\end{center}
\end{minipage}
\begin{minipage}{.5\textwidth}
\begin{center}
\begin{tikzpicture}
\draw [ultra thick,->] (0,0) -- (0.9,0);
\draw [ultra thick,->] (1,0) -- (1,0.9);
\draw [ultra thick,->] (1,1) -- (0.1,1);
\draw [ultra thick,->] (0,1) -- (0,0.1);
\draw [fill] (0,0) circle [radius=0.1];
\draw [fill] (1,1) circle [radius=0.1];
\draw [fill] (1,0) circle [radius=0.1];
\draw [fill] (0,1) circle [radius=0.1];
\node at (1,-2) {Figure 5};
\end{tikzpicture}
\end{center}
\end{minipage}
\end{myeg}

\subsection{Quasicoherent sheaves of \typea}

\begin{mydef}(Type-$\alpha$ condition for $\mathcal{O}_X$-modules)\label{conditionalpha}
Let $(X, \mc{O}_X)$ be a monoid scheme and $\mathcal{F} \in \Qcoh(X)$ . We say that $\mathcal{F}$ is \emph{type-$\alpha$} if the stalk $\F_x$ is a \emph{type-$\alpha$} $\mc{O}_{X,x}$-module for all $x \in X$. 
\end{mydef}

The following lemma shows that the type-$\alpha$ condition is affine-local. 

\begin{lem}\label{lemma: type alpha affine}
Let $(X, \mc{O}_X)$ be a monoid scheme and $\mathcal{F} \in \Qcoh(X)$ . The following conditions are equivalent:
\begin{enumerate}
	\item
$\mathcal{F}$ is type-$\alpha$.
	\item 
For every affine open neighborhood $U \subset X$, $\mathcal{F}(U)$ is a type-$\alpha$ $\mathcal{O}_X(U)$-module.	
	\item
Every $x \in X$ is contained in an open affine neighborhood $U \subset X$ such that $\F(U)$ is a type-$\alpha$ $\mc{O}_X (U)$-module. 
\end{enumerate}
\end{lem}
\begin{proof}
$(1) \implies (2)$: Suppose $U= \spec A$ for some monoid $A$, and $\mathcal{F}=\wt{M}$ for $M \in \Amod$. As explained in Remark \ref{stalkvsaffine}, $A$ has a unique maximal ideal $\mathfrak{m}=A \backslash A^{\times}$ such that $M_\mathfrak{m}=M$.  Since stalks of $\F$ are assumed \typea, $M \in \Amoda$.  

$(2) \implies (3)$: This is clear

$(3) \implies (1)$: This follows from Proposition \ref{alpha_localization}. 
\end{proof}

\begin{mydef}
Let $(X, \mc{O}_X)$ be a monoid scheme. 
\begin{enumerate}
\item 
The category $\Qcoh^{\alpha}(X)$ is the (non-full) subcategory of $\Qcoh(X)$ with objects quasicoherent sheaves satisfying Definition \ref{conditionalpha}, and morphisms inducing \typea \; morphisms on stalks. The category $\Coh^{\alpha}(X)$ is the full subcategory of $\Qcoh^{\alpha}(X)$ consisting of coherent sheaves satisfying Definition \ref{conditionalpha}. 
\item 
If $\mc{I} \subset \mc{O}_X$ is a quasicoherent sheaf of ideals,  $\Qcoh^{\alpha}(X)_{\mc{I}}$ (resp. $\Coh^{\alpha}(X)_{\mc{I}}$) is the full subcategory of $\Qcoh^{\alpha}(X)$ (resp. $\Coh^{\alpha}(X)$) consisting of those $\F$ such that $\mc{I} \cdot \F = 0$. 
\end{enumerate}
\end{mydef}

We have the following analogue of Proposition \ref{coh_pe_thm}

\begin{prop} \label{coh_pa_prop}
Let $(X, \mc{O}_X)$ be a monoid scheme.
\begin{enumerate}
\item The categories $\Qcoh^{\alpha}(X)$ and $\Coh^{\alpha}(X)$ are proto-abelian.
\item If $Z \subset X$ is a closed subset, the categories $\Qcoh^{\alpha}(X)_Z, \Coh^{\alpha}(X)_Z$ of \typea sheaves supported in $Z$ are proto-abelian. 
\item If $\mc{I} \subset \mc{O}_X$ is a quasicoherent sheaf of ideals, the categories $\Qcoh^{\alpha}(X)_{\mc{I}}$ and $\Coh^{\alpha}(X)_{\mc{I}}$ are proto-abelian. 
\item If $X = \spec A$, then the functor $$ \wt{}: \Amoda \rightarrow \Qcoh^{\alpha}(X) $$ is an exact equivalence, restricting to an exact equivalence $$ \wt{}: (\Amoda)^{fg} \rightarrow \Coh^{\alpha}(X).$$
\end{enumerate}
\end{prop}

\begin{proof}
This is proved just as Proposition \ref{coh_pe_thm}, using Lemma \ref{alpha_localization} in (4). 
\end{proof}

\begin{rmk}
In \cite{HW} the authors define a class of quasicoherent sheaves termed \emph{partially cancellative} (PC) and show that these form a \emph{quasi-exact category}. The PC condition is very similar in spirit to our \typea condition, but is essentially different, since it involves "canceling" an element of the module to deduce the equality of elements of the monoid, whereas the \typea condition involves canceling an element of the monoid to deduce equality of elements in the module. For instance, in Example \ref{alphaA1}, a rooted tree with non-trivial branching is PC but not \typea, whereas a pure cycle is \typea but not PC. 
\end{rmk}

\subsection{Ext and scalar extension}


We recall from Section \ref{pe_categories} that in any proto-exact/proto-abelian category $\on{Ext}(M,N)$ denotes the set of equivalence classes of admissible short exact sequences $\ses{N}{P}{M} $ where two such are equivalent if the diagram \eqref{cd_ses} commutes. Given a commutative ring $k$, the scalar extension functors of Section \ref{realization} are exact. Thus, given a monoid $A$, and $M, N \in \Amod$, we obtain a map a pointed sets (with base-points being the split extensions)
\begin{equation}
\E_k : \on{Ext}_{\Amod}(M,N) \rightarrow \on{Ext}^1_{k[A]} (k[M], k[N]) 
\end{equation}
sending the admissible exact sequence $\ses{N}{P}{M}$ to $ 0 \rightarrow k[N] \rightarrow k[P] \rightarrow k[M] \rightarrow 0$. Similarly, given a monoid scheme $(X, \mc{O}_X)$, and $\F, \F' \in \Coh(X)$, we obtain a map of pointed sets
\begin{equation} \label{ext_map_coh}
{\E}_k: \on{Ext}_{\Coh(X)}(\F', \F) \rightarrow \on{Ext}^1_{\Coh(X_k)}(\F'_k, \F_k).   
\end{equation}  
where we have abused notation by using $\E_k$ for both categories $\Amod$ and $\Coh(X)$. 
Our first observation is that the maps ${\E}_k$ need not be injective. 

\begin{myeg}
Let $A = \langle t \rangle$, and consider the following $A$-modules:
\begin{itemize}
\item $N=\{*, a, b \}$, with $t\cdot a = b$, $t \cdot b = *$.
\item $P= \{*, a, b, c \}$, with $t \cdot a = b$,  $t \cdot b = *$, $t \cdot c = b$.
\item $M = \{ *, c \}$ with $t \cdot c =*$.
\end{itemize}
Defining maps by the property that elements get sent to those with the same labels, we then have admissible exact sequences $$\ses{N}{P}{M}$$ and $$\ses{N}{N \oplus M}{M},$$ which are clearly not equivalent since $P$ and $N \oplus M$ are not isomorphic, $P$ being indecomposable. However, if $k$ is a field, then $ \psi_k:  k[P] \simeq k[N] \oplus k[M] $, defined by $\psi_k (a) = a, \psi_k (b)=b, \psi_k (c) = c+a$ is an isomorphism of $k[t]$-modules, which (after taking identity isomorphisms on $k[N]$ and $k[M]$) yields and isomorphism of exact sequences $ 0 \rightarrow k[N] \rightarrow k[P] \rightarrow k[M] \rightarrow 0$, and $ 0 \rightarrow k[N] \rightarrow k[N \oplus M] \rightarrow k[M] \rightarrow 0.$
\end{myeg}

The example \cite[Example 4]{lorscheid2018quasicoherent} exhibits infinitely many distinct extensions in $\Coh(\mathbb{P}^1)$ mapping to the same one on the realization $\mathbb{P}^1_k$. 

We now proceed to show that the maps $\E_k$ are injective if the modules/sheaves are of \typea. We begin with the following lemma.



\begin{lem}\label{lemma: unique lift}
Let $A$ be a monoid, and $M, N \in \Amoda$. If two exact sequences in $\Amoda$
\begin{equation} \label{twoses}
\begin{tikzcd}
0 \arrow{r}& M \arrow{r}{i} & P \arrow{r}{\pi}& N \arrow{r} & 0 \\
0 \arrow{r}& M \arrow{r}{i'} & P' \arrow{r}{\pi'}& N \arrow{r} & 0 \\
\end{tikzcd}
\end{equation}
become equivalent after scalar extension to a field $k$, then there exists a unique isomorphism $g \in \on{Hom}_{\Amoda}(P, P')$ making the diagram 
\begin{equation} \label{cdamod}
\begin{tikzcd}
0 \arrow{r}& M \arrow{dr}[swap]{i'}  \arrow{r}{i} & P \arrow{d}{g} \arrow{r}{\pi}& N  \arrow{r} & 0 \\
& & P' \arrow{ur}[swap]{\pi'}
\end{tikzcd}
\end{equation}
commute. In other words, the two sequences in \eqref{twoses} are equivalent as admissible exact sequences in $\Amoda$. 
\end{lem}
\begin{proof}
As pointed sets, we write explicitly
\[
M = \{0,  x_r \}_{r \in I}, \quad N=\{0, y_s \}_{s \in J},
\]
- that is, $x_r$ (resp.~$y_s$) are the nonzero elements in $M$ (resp.~$N$) for all $r \in I$ (resp.~$s \in J$). We identify $P = M \vee N = P'$. It follows that if $g: P \rightarrow P' $ making \eqref{cdamod} commute exists, it is uniquely (as a map of pointed sets) determined on $M, N$, and therefore unique. 

Since the exact sequences \eqref{twoses} are equivalent after base change to a field $k$, we have a commutative diagram of $k[A]$-modules:	
\begin{equation}\label{diagram: tensor with k}
\begin{tikzcd}
0 \arrow{r}& k[M]  \arrow{dr}[swap]{i'}  \arrow{r}{i} & k[P] \arrow{d}{f} \arrow{r}{\pi}& k[N]  \arrow{r} & 0 \\
& & k[P']  \arrow{ur}[swap]{\pi'}
\end{tikzcd}
\end{equation}
where we used the same letter to denote $k$-linear extensions of maps, and the map $f:k[P] \to k[P'] $ is an isomorphism of $k[A]$-modules.
We write $P = \{0, u_r, v_s \}_{r \in I,s \in J}$ with $i:M \ra P$ sending $x_r$ to $u_r$ and $\pi:P \ra N$ sending $v_s$ to $y_s$ and all the $u_r$ to $0$. Similarly, $P'=\{ 0, \ol{u}_r, \ol{v}_s \}$ with the maps $i':M \to P'$ and $\pi':P' \to N$. The unique isomorphism of pointed sets $g: P \to P'$ making \eqref{cdamod} commute then sends
$u_r$ to $\ol{u}_r$ and $v_s$ to $\ol{v}_s$. As $k$-vector spaces (not $k[A]$ -modules), we identify $k[P] = k[M] \oplus k[N]$, with the first summand spanned by $u_r$'s and the second by $v_s$'s, and similarly with $k[P']$. 

Since $f: k[P]  \ra k[P'] $ is a $k[A]$-isomorphism inducing the identity on $k[M]$ and $k[N]$, we must have 
\begin{equation}\label{eq: f maps to}
f(u_r) = \ol{u}_r, \textrm{ and } f(v_s) = \ol{v}_s + \ol{L}(v_s), 
\end{equation}
where $\ol{L}(v_s) \in k[M]$ (i.e. it is a linear combination of $\ol{u}_r$'s). 
We now use the fact that $f$ in \eqref{diagram: tensor with k} is a $k[A]$-module isomorphism and $M,N,P, P' \in \Amoda$ to deduce that $g$ is an $A$-module isomorphism. Since $g$ is a bijection, this amounts to checking it is $A$-equivariant, i.e.
\begin{equation}\label{eq: equivariant}
g(ap)=ag(p), \quad \forall a \in A,~p \in P. 
\end{equation}

The case when $p=0$ is clear, and hence there are four cases to consider:
\begin{enumerate}
	\item[(i)] 
	$p = u_r$ for some $r$. 
	\item[(ii)] 
	$p = v_s$, $ap = u_n \neq 0$ for some $s, n$. 
	\item[(iii)] 
	$p = v_s$, $ap = v_t \neq 0$ for some $s,t$.
	\item[(iv)] 
	$p= v_s$, $ap = 0$. 
\end{enumerate}

We prove that in each case \eqref{eq: equivariant} holds.

\begin{enumerate}
	
\item[(i)] 
The first case is trivial since $f$ and $g$ agree on elements in $M$, and if $p \in M$ (by considering $M$ as a subset of $P$), then so is $ap$.
\item[(ii)] 
In this case, one may observe that $f(ap) = a f(p)$ implies the following equation:
	\begin{equation} \label{eq(1)}
	a \ol{v}_s + a \ol L(v_s) = \ol{u}_n.
	\end{equation}
	We have from \eqref{eq: f maps to} that 
	\[
	\ol{v}_s = f(v_s - L(v_s)),
	\]
	where $L(v_s) = f^{-1}(\ol{L}(v_s))$. In concrete terms, it is the same linear combination of $u$'s as $\ol{L}(v_s)$ is of $\ol{u}$'s.
	We first show that $a \ol{v}_s \neq 0$, for if $a \ol{v}_s = 0$, 
\[
0 = a \ol{v}_s = a f(v_s - L(v_s)) = f(a(v_s - L(v_s))) 
\]	
	
and since $f$ is an isomorphism, we have that
\begin{equation}\label{eq: last equation}
a v_s - a L(v_s) = 0.
\end{equation}
From the assumption, we know $a v_s = u_n \neq 0$, and so for \eqref{eq: last equation} to hold, there is a cancellation between this term and some term in $a L(v_s)$, which is a linear combination of $au$'s. 
For this to happen, we would have to have $a v_s = a u_j \neq 0$ for some $j$, but from the type-$\alpha$ condition this is impossible since these are distinct. This proves our claim.

Returning to \eqref{eq(1)}, knowing that $a \ol{v}_s \neq 0$ the same kind of argument shows that unless $a \ol{v}_s = \ol{u}_n$, $a \ol{v}_s$ must cancel with some term occurring in $a \ol{L}(v_s)$, which again by the type-$\alpha$ condition is impossible. In particular, this implies that
\[
a g(v_s)=  a \ol{v}_s = \ol{u}_n = g(u_n) = g( a v_s).
\]
\item[(iii)] 
	In the third case, $f(ap) = af(p)$ implies the equality 
	\[
	\ol{v}_t + \ol L(v_t) = a \ol{v}_s + a \ol L(v_s).
	\]
	If $a \ol{v}_s \in M \subset k[M] $, then $\ol{v}_t \in k[M]$. In particular, $\ol{v}_t$ is a linear combination of $\ol{u}$'s, which is impossible. It follows that $a \ol{v}_s \in N \subset k[N]$, and hence we have that
	\[
	\ol{v}_t - a \ol{v}_s = a \ol{L}(v_s) - \ol{L}(v_t),
	\] 
	where the LHS lies in $k[N]$ and the RHS in $k[M]$. Thus, both LHS and RHS are zero.  In particular, $\ol{v}_t = a \ol{v}_s$, and hence
	\[
	g(a v_s) = g(v_t) = \ol{v}_t = a \ol{v}_s = a g(v_s).
	\]
	\item[(iv)] 
	In this case, we have that $$ a \ol{v}_s + a \ol{L}(v_s) = 0, $$ and if $a \ol{v}_s \neq 0$, it must cancel a term in $a \ol{L}(v_s)$ which is impossible by the type-$\alpha$ condition. Therefore, we have that $g(av_s)=0=a\ol{v}_s=ag(v_s)$. 
\end{enumerate}
This completes the proof.
\end{proof}

We now use this result to deduce that the map \eqref{ext_map_coh} is an injection for any field $k$. 

\begin{prop} \label{ext_scalar_extension}
Let $(X, \mc{O}_X)$ be a monoid scheme, $k$ a field, and $\mc{F}, \mc{F}' \in \Coh^{\alpha}(X)$. Then the map
\[
\E_k: \on{Ext}_{\Coh^{\alpha}(X)}(\F', \F) \rightarrow \on{Ext}^1_{\Coh(X_k)}(\F'_k, \F_k)   
\]
is an injection.
\end{prop}

\begin{proof}
Let $\ses{\F'}{\G}{\F}$ and $\ses{\F'}{\G'}{\F}$ be two admissible exact sequences in $\Coh^{\alpha}(X)$ whose images under $\E_k$ are equivalent. The same line of argument as given at the beginning of the proof of Lemma  \ref{lemma: unique lift} shows that there is a unique isomorphism of sheaves of pointed sets $\phi: \G \to \G'$ (not a priori of $\mc{O}_X$-modules), making the diagram

\begin{equation} \label{cdsheaves}
\begin{tikzcd}
0 \arrow{r}& \F \arrow{dr}{}  \arrow{r}{} & \G \arrow{d}{\phi} \arrow{r}{}& \F'  \arrow{r} & 0 \\
& & \G' \arrow{ur}{}
\end{tikzcd}
\end{equation}
commute. We want to show that $\phi$ is $\mc{O}_X$-equivariant. 

Let $U \subset X$ be open, $\{ U_{i} \}_{i \in I}$ be an affine open cover of $U$. Let $a \in \mc{O}_X (U)$ and $\mu_a:\mc{G}_k|_U \to \mc{G}_k|_U$ multiplication by $a$. Consider the map of sheaves of $k$-vector spaces
\[
\psi:=\mu_a \circ \phi_k|_U - \phi_k|_U\circ \mu_a \in Hom_{Sh_{k}}(\G_k \vert_U, \G'_k \vert_U),
\]
where $\phi_k$ denotes the $k$-linear extension of $\phi$. 
By Lemma \ref{lemma: unique lift}, $\psi \vert_{U_{i}} = 0$ for every $i \in I$. This means that $\phi(U)(a \cdot s) = a \phi(U)( s) $ for every $s \in \G(U)$, showing that $\phi$ is an isomorphism in $\Coh^{\alpha}(X)$. This proves that $\E_k$ is injective. 
\end{proof}

We now arrive at our main result:

\begin{mythm} \label{finitary_alpha}
Let $\Delta$ be the fan of a projective toric variety, and $(X_{\Delta}, \mc{O}_X)$ the corresponding monoid scheme. Then
\begin{enumerate}
\item $\Coh^{\alpha}(X_{\Delta})$ is a finitary proto-abelian category.  
\item If $Z \subset X_{\Delta}$ is a closed subset, the category $\Coh^{\alpha}(X_{\Delta})_Z$ of \typea sheaves supported in $Z$ is a finitary proto-abelian category. 
\item If $\mc{I} \subset \mc{O}_X$ is a quasicoherent sheaf of ideals, the category $\Coh^{\alpha}(X_{\Delta})_{\mc{I}}$ is a finitary proto-abelian category. 
\item The Hall algebra $H_{\Coh^{\alpha}(X_{\Delta})}$ is isomorphic, as  a Hopf algebra, to an enveloping algebra $U(\mathfrak{n}^{\alpha}_X)$, where $\mathfrak{n}^{\alpha}_X$ has the indecomposable sheaves in $\Coh^{\alpha}(X_{\Delta})$ as a basis. 
\end{enumerate}
\end{mythm}

\begin{proof}
For the first part, let $k$ be a finite field, and let $\F, \F' \in \Coh(X_{\Delta})$. Then $\F_k, \F'_k \in \Coh(X_{\Delta,k})$,  and since $X_{\Delta, k}$ is a projective variety over $k$, it is well-known that $\on{Hom}_{\mc{O}_{X_{\Delta,k}}}(\F_k, \F'_k)$ and $\on{Ext}^1_{\mc{O}_{X_{\Delta,k}}}(\F_k, \F'_k) $ are finite-dimensional vector spaces over $k$, hence finite sets. The scalar extension map $$\on{Hom}_{\Coh^{\alpha}_{X_{\Delta}}}(\F, \F') \rightarrow \on{Hom}_{\Coh^{\alpha}_{X_{\Delta,k}}}(\F_k, \F'_k)$$ is injective, showing that the LHS is finite. By Proposition \ref{ext_scalar_extension}, $ \on{Ext}_{\Coh^{\alpha}(X_{\Delta})}(\F', \F)$ is finite as well. This shows that $\Coh^{\alpha}(X_{\Delta})$ is finitary, and proto-abelian by Proposition \ref{coh_pa_prop}. 

The second and third parts follow from the fact that $\Coh^{\alpha}(X_{\Delta})_Z$,  $\Coh^{\alpha}(X_{\Delta})_{\mc{I}}$ are full subcategories of $\Coh^{\alpha}(X_{\Delta})$ and the first part. 
The last part follows from Proposition \ref{Hall_theorem}.
\end{proof}

To simplify notation, we will write
\[
H^{\alpha}_X := H_{\Coh^{\alpha}{(X)}}
\]
whenever it's defined. 

\subsection{The Hall algebra of $\Pone$} \label{Pone}

We review here the example of the Hall algebra $H^{\alpha}_{\Pone}$ of the category $\Coh^{\alpha}(\Pone)$ considered in \cite{Sz1}, where we refer the interested reader for details. It is shown in \cite{Sz1} that the indecomposable objects of $\Coh^{\alpha}(\Pone)$ are:

\begin{itemize}
\item 
The line bundles $\mc{O}(n)$ (see Example \ref{P1sheaves}).
\item 
The torsion sheaves $\mc{T}_{0, m}$, $\mc{T}_{\infty, m}$ (see Example \ref{P1sheaves}).
\item 
``Cyclotomic'' sheaves $\mc{C}_l$ represented in notation of Example \ref{P1sheaves} as $(M, M', \phi)$, where each of the associated directed graphs $\Gamma_M, \Gamma_{M'}$ is a pure cycle of length $l$ (see Figure 5 of Example \ref{alphaA1}), and $\phi: M_t \rightarrow M'_{t^{-1}}$ is any isomorphism (all choices of such a $\phi$ yields isomorphic sheaves).  
\end{itemize}

By Proposition \ref{Hall_theorem}, the Hall algebra $H^{\alpha}_{\Pone}$ is isomorphic to an enveloping algebra $U(\n^{\alpha})$, where $\n^{\alpha}$ has as a basis the isomorphism classes the above indecomposable sheaves. The commutation relations among the basis $\mc{O}(n), \mc{T}_{x,m}, \mc{C}_k$ are (see \cite{Sz1}): 
$$
[\mc{T}_{x,m}, \mc{O}(n)] = \mc{O}(m+n)
$$
with all other commutators $0$. 

Let $$ h_1 = \begin{pmatrix} 1 & 0 \\ 0 & 0 \end{pmatrix}, \quad  h_2 = \begin{pmatrix} 0 & 0 \\ 0 & 1 \end{pmatrix},\quad  e = \begin{pmatrix} 0 & 1 \\ 0 & 0 \end{pmatrix},$$ and let $\mathfrak{b}$ be the Lie subalgebra of $\mathfrak{gl}_2[t,t^{-1}]$ with basis $\{ h_1 \otimes t^r, h_2 \otimes t^s, e \otimes t^n \}, r,s \geq 1, n \in \mathbb{Z}$. Let $\mathfrak{a}$ denote the Lie subalgebra of $\mathfrak{n}^{\alpha}$ spanned by $\{ \mc{T}_{0, r}, \mc{T}_{\infty,s}, \mc{O}(n) \}, r, s \geq 0, n \in \mathbb{Z}$. We have an isomorphism
 $$ \rho: \mathfrak{a} \rightarrow \mathfrak{b} $$
 \begin{align*} \rho(\mc{T}_{0, r}) &= h_1 \otimes t^r  \\ \rho(\mc{T}_{\infty, s}) &= -h_2 \otimes t^s \\ \rho(\mc{O}(n)) &=e \otimes t^n \end{align*} The Lie algebra $\mathfrak{b}$ can be viewed as a non-standard Borel subalgebra in $\mathfrak{gl}_2 [t,t^{-1}]$. 
Since the generators $\mc{C}_l$ are central, we have a splitting $\mathfrak{n}^{\alpha} \simeq \mathfrak{a} \oplus \mathfrak{k}$, where $\mathfrak{k}$ is an abelian Lie subalgebra on the countably many generators $\mc{C}_1, \mc{C}_2, \cdots$. To summarize

\begin{prop} \label{Pone_alpha}
$H^{\alpha}_{\Pone}$ is isomorphic as a Hopf algebra to $U(\mathfrak{a} \oplus \mathfrak{k})$, where $\mathfrak{a}$ is isomorphic to a subalgebra of $\mathfrak{gl}_2 [t, t^{-1}]$, and $\mathfrak{k}$ is an abelian Lie algebra with generators $\mc{C}_1, \mc{C}_2, \cdots$. 
\end{prop}

For fixed $x \in \{ 0, \infty \}$, the torsion sheaves $\mc{T}_{x, r}, r \geq 1$ supported at $x$ multiply according to:
$$ \mc{T}_{x, r} \bullet \mc{T}_{x,s} = \mc{T}_{x, r+s} + \mc{T}_{x,r} \oplus \mc{T}_{x,s} $$
The Hopf subalgebra $H^{\alpha}_x$ of $H^{\alpha}_{\Pone}$ generated by these is therefore isomorphic to the Hopf algebra $\Lambda$ of symmetric functions, via the map sending $\mc{T}_{x,r}$ to $p_r$ - the $r$th power sum. 

\section{ The category of T-sheaves}

Let $X$ be a monoid scheme. In this section we define a full subcategory $\Coh^T(X) \subset \Coh^{\alpha}(X)$ of  \emph{coherent $T$-sheaves} on $X$. It is the sub-category of  $\Coh^{\alpha}(X)$ consisting of sheaves locally isomorphic to direct sums of quotients of ideal sheaves.

\begin{mydef} \label{Tsheaf}

$\F \in \Coh^{\alpha}(X)$ is a \emph{$T$-sheaf} if there is a cover of $X$ by open affines $\{ U_j \}^{r}_{j=1}$ such that for each $j=1, \cdots, r$
\[
\F \vert_{U_j} \simeq \oplus^{k_j}_{i=1} \mc{I}_i/\mc{J}_i
\]
where $\mc{J}_i \subset \mc{I}_i \subset \mc{O}_{X}$ are ideal sheaves. Equivalently, 
\[
\F \vert_{U_j} \simeq \oplus^{k_j}_{i=1} \wt{I_i/J_i}
\]
where $J_i \subset I_i \subset \mc{O}_X (U_j)$ are ideals. 
\end{mydef}

It follows from the remarks following Definition \ref{typeamod} that any $T$-sheaf is an object in  $\Coh^{\alpha}(X)$. We denote by $\Coh^T(X)$ the full subcategory of $\Coh^{\alpha}(X)$ whose objects are $T$-sheaves. 

\begin{rmk}

We note that $\Coh^T(X)$ contains all invertible sheaves. Moreover, when $X$ is separated, Theorem 3.3 of \cite{Pir} shows that any vector bundle is a direct sum of line bundles, and so $\Coh^T(X)$ contains all such.  

\end{rmk}

The following is proved just as Lemma \ref{lemma: type alpha affine}:

\begin{lem}\label{lemma: Tsheaf affine}
Let $\F \in \Coh^{\alpha}(X)$. The following conditions are equivalent:
\begin{enumerate}
	\item
$\mathcal{F} \in \Coh^T(X)$.
	\item 
For every affine open neighborhood $U \subset X$, $\mathcal{F} \simeq \oplus^k_{i=1} \wt{I_i/J_i} $ where $J_i \subset I_i \subset \mathcal{O}_X(U)$ are ideals. 	
	\item
For every $x \in X$, $\mathcal{F}_x \simeq \oplus^k_{i=1} \mc{I}_{i,x}/ \mc{J}_{i,x} $ where $\mc{J}_i \subset \mc{I}_i \subset \mathcal{O}_X$ are coherent ideal sheaves. 
\end{enumerate}
\end{lem}

We have the following analogue of Proposition \ref{coh_pa_prop}:

\begin{prop} \label{pa_t}
	Let $X$ be a monoid scheme. 
	\begin{enumerate}
	\item $\Coh^T(X)$ is a full proto-abelian subcategory of  $\Coh^{\alpha}(X)$.
	\item If $Z \subset X$ is a closed subset, the category $\Coh^T(X)_Z$ of $T$-sheaves supported in $Z$ is a full proto-abelian, extension-closed subcategory of  $\Coh^T(X)$.
	\item If $\mc{I} \subset \mc{O}_{X}$ is a quasicoherent sheaf of ideals, then the category $\Coh^T(X)_{\mc{I}}$ is a full proto-abelian subcategory of $\Coh^T(X)$.
	\end{enumerate}
\end{prop}
\begin{proof}
$(1)$: it follows from Remark \ref{pe_structure_Amod} that any sub-sheaf or quotient of a $T$-sheaf is a $T$-sheaf, proving the claim. For $(2)$ and $(3)$, one may apply the same argument as in Proposition \ref{coh_pe_thm}.
\end{proof}

\subsection{Local picture - $T$-sheaves on $\mathbb{A}^n$ } 

Let $A = \Pn$. By Lemma \ref{lemma: Tsheaf affine}, every T-sheaf on $\mathbb{A}^n = \mspec(A)$ is of the form $\oplus^k_{i=1} \wt{I_i/J_i}$ where $J_i \subset I_i \subset A$ is a chain of ideals. $A \setminus \{0_A\}$ may be given the structure of a poset, where $a \geq b$ if $a \vert b$. By taking the vector of exponents of each of the variables $x_i$, this poset may be identified with $(\mathbb{Z}_{\geq 0}^n, \geq')$ where $u \geq' v$ if and only if  $v= u+w$ for some $w$. We will use this identification throughout. 

\begin{mydef}
A \emph{generalized skew shape} $\T$ in $A$ is a convex sub-poset $\T \subset (A\setminus \{0_A\}, \geq)$.  We consider two such to be equivalent if they are isomorphic as posets. 
We say that $\T$ is \emph{connected} if it is connected as a poset (i.e. if the corresponding Hasse diagram is connected). 
\end{mydef}

After deleting the zero element, to each chain of ideals $J \subset I \subset A$ corresponds a chain of convex subposets $P_J \subset P_I \subset A \setminus  0_A $, with $P_J$ downward closed in $P_I$. To the quotient $I/J$ corresponds the complement $P_I \setminus P_J$, which is also convex. Thus, to each quotient of ideals $I/J$ we attach a generalized skew shape denoted $\T_{I/J}$. 

\begin{rmk}
Note that if $a \in A \in \setminus \{0_A\} $, then as $A$-modules, $ I/J \simeq aI /aJ$, with isomorphism given by multiplication by $a$. On the poset side, the posets $P_{aJ} \subset P_{aI}$ are obtained by translating $P_J \subset P_I$ by the exponent vector of $a$, and so $\T_{aI/aJ}$ and $\T_{I/J}$ are isomorphic as posets. 
\end{rmk}

The terminology \emph{skew shape} stems from the fact that when $n=2$, i.e. $A = \fun \langle x_1, x_2 \rangle $, the equivalence classes of finite $\T$'s correspond to skew Young diagrams, as seen in the following example:

\begin{myeg}[\cite{Sz0}] \label{example_a}
Let $n=2$, $J=(x^4_1, x^2_1 x_2, x_1 x^2_2, x^3_2 )  \subset I = (x_1, x_2) \subset A = \fun \langle x_1, x_2 \rangle $. Then $I/J = \{x_1, x^2_1, x^3_1, x_2, x_1 x_2, x^2_2 \} $, and $\T_{I/J}$ corresponds to the connected skew Young diagram:  

\begin{center}
\Ylinethick{1pt}
\gyoung(;,;;,:;;;)
\end{center}
\end{myeg}
\vspace{0.1cm}

Here, the missing box in the lower left hand corner corresponds to $1$, the top left box to $x^2_2$, and the bottom right to $x^3_1$.


\begin{rmk}
The poset $\T_{I/J}$ is connected iff the corresponding skew Young diagram is connected - i.e. if any two boxes in the diagram may be joined by a path staying within the diagram. 
\end{rmk}

The following is straightforward:

\begin{prop} \label{ssm_prop}
Let $J \subset I \subset  A \simeq \fun\langle x_1,\dots,x_n\rangle$ be a chain of ideals, and $\T_{I/J} \subset (A\setminus \{0_A\}, \geq) $ the generalized skew shape corresponding to $I/J$, then:
\begin{enumerate}
\item $(\T, \geq)$ has a finite number of maximal elements, which correspond to the unique minimal set of generators of $I/J$ as an $A$-module.
\item $A$-submodules of $I/J$ correspond to downward closed subsets in $\T_{I/J}$ 
\item $I/J$ is indecomposable as an $A$-module iff $\T_{I/J}$ is connected. 
\end{enumerate}
\end{prop}

\begin{rmk}
Since an $A$-module of the form $I/J$ may be naturally identified with a subset of $A$, it admits a grading by $A$. This property was used in \cite{Sz0} to define T-sheaves supported at $0$ in $\mathbb{A}^n$ and their Hall algebra. Definition \ref{Tsheaf} above seems to us more natural and easier to work with than that given in \cite{Sz0}. 
\end{rmk}

As a converse to part $(3)$ of Prop. \ref{ssm_prop}, to a connected generalized skew shape $\T \subset (A\setminus \{0_A\}, \leq)$ one can attach an indecomposable $A$-module $M_{\T}$, with underlying set $\T \sqcup \{ 0 \}$, and $A$-module structure given by: for $a \in A$ and $t \in {\T}$,
\begin{align*}
a \cdot t &= \begin{cases} at, & \mbox{if } at \in \T \\ 0 & \mbox{otherwise}. \end{cases} 
\end{align*}
We have $M_{\T} \simeq I/J$ where $I$ is the ideal generated by $\T$, and $J= I \setminus \T$. 

\begin{myeg} Let $\T$ as in Example \ref{example_a}. In terms of the skew Young diagram,  $x_1$ (resp.~$x_2$) acts on $M_{\T}$ by moving one box to the right (resp.~one box up) until reaching the edge of the diagram, and $0$ after that. A minimal set of generators for $M_{\T}$ is indicated by the black dots:

\begin{center}
\Ylinethick{1pt}
\gyoung(;,\bullet;,:;\bullet;;)
\end{center}
\end{myeg}

\begin{myeg}[\cite{Sz0}]  Let $A = \fun \langle x_1, x_2 \rangle$. Let $\S$ be the generalized skew shape below, with $\T \subset \S$ the downward closed subset consisting of the boxes containing t's.

\begin{center}
\Ylinethick{1pt}
\gyoung(;t,;;;t,:;;t;t,::;;;t)
\end{center}
\Ylinethick{1pt}
We have the following decomposition of $M_{\T}$
\[
M_{\T} = \gyoung(;) \oplus \gyoung(;) \oplus \gyoung(;,;;)
\]
and the quotient
\[
M_{\S} / M_{\T} = \gyoung(;;,:;) \oplus \gyoung(;;)
\]
One can see that $\textrm{Ann}_A(M_{\T})=(x_1^2, x_1 x_2, x_2^2)$, and hence $\textrm{supp}(M_{\T})=V(x_1,x_2)=0 \subset \mathbb{A}^2$. 
\end{myeg}

\begin{myeg}
Let $A = \fun \langle x_1, x_2 \rangle$, and consider the following subset $\T \subset (A\setminus \{0_A\}, \leq)$ 
\[
\T:=\{ x^{n_1}_1 x^{n_2}_2  \mid n_1 \in \mathbb{N}, 0 \leq n_2 < 2 \} \cup \{ x^{n_1}_1 x^{n_2}_2  \mid n_2 \in \mathbb{N}, 0 \leq n_1 < 2  \}.
\]
$\T$ corresponds to the infinite shape

\begin{center}
\Ylinethick{1pt}
\gyoung(:\vdots:\vdots,;;,;;,;;,;;;;;:\cdots,;;;;;:\cdots)
\end{center}
\Ylinethick{1pt}
\vspace{0.1cm}

We have $M_{\T} \simeq A/(x^2_1 x^2_2)$, and $\textrm{Ann}_A(M_{\T})=(x^2_1 x^2_2)$. Hence $\textrm{supp}(M_{\T})=V(x_1x_2)\subset \mathbb{A}^2$ is the union of the $x_1$ and $x_2$-axes. 

\end{myeg}

\begin{myeg} \label{tf_example}
Let $A = \fun \langle x_1, x_2 \rangle$ and $\T \subset (A\setminus \{0_A\}, \leq)$ be the ideal $\T:=(x^4_1, x_1 x_2, x^3_2 )$. $\T$ corresponds to the infinite shape

\begin{center}
\Ylinethick{1pt}
\gyoung(:\vdots:\vdots:\vdots:\vdots:\vdots:\vdots:\iddots,;;;;;;:\cdots,;;;;;;:\cdots,:;;;;;:\cdots,:;;;;;:\cdots,::::;;:\cdots)
\end{center}
\Ylinethick{1pt}
\vspace{0.1cm}
Viewing $\T$ as an ideal in $A$, we have $M_{\T} \simeq {\T}$.  One observes that $\textrm{Ann}_A(M_{\T})=\{0_A\}$, and hence $\textrm{supp}(M_{\T})=\mathbb{A}^2$. 
\end{myeg}


\begin{rmk}
If $\T$ is a finite skew shape, then for sufficiently large $k$, $\mathfrak{m}^k \subset \on{Ann}(M_{\T})$, where $\mathfrak{m}=( x_1, \cdots, x_n) $ is the maximal ideal of $\A$. Thus $\textrm{supp}(M_{\T}) = 0 \subset \mathbb{A}^n$. 
\end{rmk}

The following summarizes our discussion of T-sheaves on $\mathbb{A}^n$:

\begin{prop} \label{skew_shape_prop}  
Every $\F \in \Coh(\mathbb{A}^n)^T$ can be written as a finite direct sum 
\begin{equation} \label{decomposition}  \F \simeq \F_1 \oplus \cdots \oplus \F_m ,
\end{equation} 
where $\F_{i} \simeq \wt{M_{\T_i}}, i=1 \cdots m$, for connected generalized skew shapes $\T_i \subset \Pn \setminus \{0\}$. The decomposition \eqref{decomposition} is unique up to permutation of factors. 
\end{prop}

\subsection{The Hall algebra of $T$-sheaves}

Let $\Delta$ be the fan of a projective toric variety. Since $\Coh^T(X_{\Delta})$ is a full  proto-abelian sub-category of $\Coh^{\alpha}(X_{\Delta})$, and the latter is finitary by Theorem \ref{finitary_alpha}, it follows that so is the former. We thus have the following analogue of Theorem \ref{finitary_alpha}:

\begin{mythm} \label{finitary_T}
Let $\Delta$ be the fan of a projective toric variety, and $(X_{\Delta}, \mc{O}_X)$ the corresponding monoid scheme. Then
\begin{enumerate}
\item $\Coh^T(X_{\Delta})$ is a finitary proto-abelian category.  
\item If $Z \subset X_{\Delta}$ is a closed subset, the category $\Coh^{T}(X_{\Delta})_Z$ of $T$-sheaves supported in $Z$ is a finitary proto-abelian category. 
\item If $\mc{I} \subset \mc{O}_X$ is a quasicoherent sheaf of ideals, the category $\Coh^{T}(X_{\Delta})_{\mc{I}}$ is a finitary proto-abelian category. 
\item The Hall algebra $H_{\Coh^T(X_{\Delta})}$ is isomorphic, as  a Hopf algebra, to an enveloping algebra $U(\mathfrak{n}^{T}_X)$, where $\mathfrak{n}^{T}_X$ has the indecomposable sheaves in $\Coh^{T}(X_{\Delta})$ as a basis. 
\end{enumerate}
If $Z \subset X_{\Delta}$ is a closed subset, or $\mc{I} \subset \mc{O}_{X_{\Delta}}$ is a quasicoherent sheaf of ideals, we may also define the Hall algebras of $\Coh^T(X_{\Delta})_Z$ and $\Coh^T(X_{\Delta})_{\mc{I}}$. 
\end{mythm}

To simplify notation, we will write
\[
H^{T}_X := H_{\Coh^{T}(X_{\Delta})}, \; \; H^T(X)_Z := H_{\Coh^{T}(X_{\Delta})_Z}, \; \; H^T(X)_{\mc{I}} := H_{\Coh^{T}(X_{\Delta})_{\mc{I}}}
\]

\section{Examples and relations with known Lie algebras}

In this section we attempt to explicitly describe the Hall algebra of $\Coh^T(X)$, and in some cases categories $\Coh^T(X)_Z, \Coh^T(X)_{\mc{I}}$, when $X$ is a toric monoid scheme such that $\dim(X_k) > 1$. We want to show that in contrast to the setting of projective schemes over $\mathbb{F}_q$, computations for higher-dimensional examples in the monoid setting are not only possible, but relatively straightforward. As our Hall algebras are always isomorphic to enveloping algebras, it suffices to describe the corresponding Hall Lie algebra. 

\subsection{Point sheaves on $\mathbb{A}^n$ and skew shapes} \label{skew_shapes}

The category $\Coh^T(\mathbb{A}^n)_0$ of $T$-sheaves supported at the origin in $\mathbb{A}^n$ may be identified with the category $\Coh^T(\mathbb{P}^n)_x$, where $x \in \mathbb{P}^n$ is a closed point, since $x$ has a (unique) open affine neighborhood isomorphic to $\mathbb{A}^n$. By Theorem \ref{finitary_T},  $\Coh^T(\mathbb{A}^n)_0$ is finitary (though note that $\Coh^T(\mathbb{A}^n)$ is not !), and we consider the corresponding Hall algebra $H^T(\mathbb{A}^n)_0$, first investigated by the second author in \cite{Sz0}, where we refer the interested reader for details. The examples in this section are taken from that paper. 

By part $(4)$ of Theorem \ref{finitary_T} and Proposition \ref{skew_shape_prop},  $H^T(\mathbb{A}^n)_0 \simeq \mathbf{U}(\mathfrak{sk}_n)$, where $\mathfrak{sk}_n$ is the Lie algebra with basis delta-functions supported on \emph{finite} connected generalized skew shapes. 

The symmetric group on $S_n$ acts on $\Pn$ by $\sigma\cdot x_i=x_{\sigma(i)}$ for $\sigma \in S_n$, and so on $\mathbb{A}^n$, fixing the origin. This yields an action of $S_n$ on $\Coh^T(\mathbb{A}^n)_0$ by exact auto-equivalences, and hence on  $H^T(\mathbb{A}^n)_0$ and $\mathfrak{sk}_n$. To summarize:

\begin{mythm}[\cite{Sz0}] \label{skew_Hall_theorem}
The Hall algebra $H^T(\mathbb{A}^n)_0$ is isomorphic to the enveloping algebra $\mathbf{U}(\mathfrak{sk}_n)$. The Lie algebra $\mathfrak{sk}_n$ may be identified with
\[
\mathfrak{sk}_n = \{ \delta_{[M_{\S}]} \mid \S \textrm{ a finite connected skew shape in } \Pn \setminus \{0\} \textrm{ up to translation } \}
\]
with Lie bracket 
\[
[\delta_{[M_{\S}]}, \delta_{[M_{\T}]}] = \delta_{[M_{\S}]} \bullet \delta_{[M_{\T}]} - \delta_{[M_{\T}]} \bullet \delta_{[M_{\S}]}
\]
The symmetric group $S_n$ acts on $H^T(\mathbb{A}^n)_0$ (resp.~$\mathfrak{sk}_n$) by Hopf (resp.~Lie) algebra automorphisms. 
\end{mythm}

We recall from Section \ref{subsection: Hall algebras of finitary proto-exact} that the product in $H^T(\mathbb{A}^n)_0$ is
\begin{equation} \label{skew_mult}
\delta_{[M_{\S}]} \bullet \delta_{[M_{\T}]} = \sum_{ \mc{R} \textrm{ a skew shape } } \mathbf{P}^{\mc{R}}_{\S,\T} \delta_{[M_{\mc{R}}]} 
\end{equation}
where 
\[
\mathbf{P}^{\mc{R}}_{\S,\T}  := \#  \{ M_\mc{L} \subset M_\mc{R} \mid  M_\mc{L} \simeq M_\mc{T}, M_\mc{R} / M_\mc{L} \simeq M_\mc{S} \}.
\]
Note that the skew shapes $\mc{R}$ being summed over need not be connected. In fact, when $\S$ and $\T$ are connected, the product (\ref{skew_mult}) will contain precisely one disconnected skew shape $\S \sqcup \T$ corresponding to the split extension $M_{\S} \oplus M_{\T}$. 

\begin{myeg}[\cite{Sz0}] 
Let $n=2$. By abuse of notation,  we denote the delta-function $\delta_{[M_{\S}]} \in \mathfrak{sk}_2$ simply by $\S$. Let
\Ylinethick{1pt}
\[
\S = \gyoung(;,;;) \; \;  \; \; \;\T = \gyoung(;;)
\]
Then, we have
\[
\S \bullet \T = \gyoung(;s,;s;s;t;t) \; + \; \gyoung(;t;t,:;s,:;s;s) \; + \; \gyoung(;s,;s;s) \oplus \gyoung(;t;t) 
\]
\[
\T \bullet \S = \gyoung(;s,;s;s,;t;t) \; + \; \gyoung(;s,;s;s,:;t;t) \; + \; \gyoung(;t;t;s,::;s;s) \; + \; \gyoung(;s,;s;s) \oplus \gyoung(;t;t) 
\]
and
\[
[\S,\T] =  \gyoung(;s,;s;s;t;t) \; + \; \gyoung(;t;t,:;s,:;s;s) \; - \; \gyoung(;s,;s;s,;t;t) \; - \; \gyoung(;s,;s;s,:;t;t) \; - \; \gyoung(;t;t;s,::;s;s)
\]
where for each skew shape we have indicated which boxes correspond to $\S$ and $\T$.
\end{myeg}

As the above example illustrates, for connected skew shapes $\S,\T$, the multiplication $\S \bullet \T$ in the Hall algebra involves all ways of "stacking" the shape $\T$ onto that of $\S$ to obtain another skew shape, as well as one disconnected skew shape (which may be drawn in several ways). 

\begin{rmk}
One may easily observe that the structure constants of the Lie algebra $\mathfrak{sk}_n$ in the basis of skew shapes are all $-1$, $0$, or $1$. 
\end{rmk}

The co-multiplication \eqref{coprod-def} on $H^T(\mathbb{A}^n)_0$ can be explicitly described as follows. Suppose that $\S=\S_1 \sqcup \S_2 \cdots \sqcup \S_k$ with $\S_i$ connected for each $i=1,\dots, k$. Then 
\begin{equation}
\Delta(\delta_{[M_{\S}]}) = \sum_{I \subset \{ 1, \cdots, k \} } \delta_{[M_{\S_{I}}]} \otimes \delta_{[M_{\S_{I^c}}]}
\end{equation}
where the summation is over all subsets $I \subset \{ 1, \cdots, k \}$, $I^c =  \{ 1, \cdots, k \} \backslash I $, and $$ M_{\S_I} := \bigoplus_{i \in I} M_{\S_i} .$$

\begin{myeg}[\cite{Sz0}]
Adopting the same conventions as in the previous example, we have
\Ylinethick{1pt}
\[
\Delta \left(  \gyoung(;,;;) \oplus \gyoung(;;) \right) 
\]
\[
= \left( \gyoung(;,;;) \oplus \gyoung(;;) \right) \otimes 1 + 1 \otimes \left( \gyoung(;,;;) \oplus \gyoung(;;) \right) + \gyoung(;,;;) \otimes \gyoung(;;) + \gyoung(;;) \otimes  \gyoung(;,;;)
\]
\end{myeg}

\medskip

\subsection{Point sheaves on the second infinitesimal neighborhood of the origin in $\mathbb{A}^2$}

Let $\m = (x_1, x_2)  \subset \Pt$ be the maximal ideal of the origin in $\mathbb{A}^2$. In this section we consider the category $\Coh^T(\mathbb{A}^2)_{\m^2}$ of coherent $T$-sheaves  scheme-theoretically supported on the second formal neighborhood of the origin. Embedding $\mathbb{A}^2 \subset \mathbb{P}^2$,  we may equivalently view these as sheaves supported on the second infinitesimal neighborhood of one of the three closed points.  In terms of modules, it is the category of \typea $\Pt$-modules admitting a grading which are annihilated by $\m^2 = (x^2_1, x_1 x_2, x^2_x)$. These in turn correspond to two-dimensional skew shapes not containing any of the following three as sub-diagrams:
\Ylinethick{1pt}
\[
\gyoung(;;;) \; \; \; \;  \gyoung(;,;,;) \; \; \; \; \gyoung(;;,;;)
\]
The only connected skew shapes having this property are easily seen to be of four types (depending on whether the first/last part consists of one or two boxes):
\[
\gyoung(;,;;,:;;,:::\vdots,:::;;,::::;;,:::A) \; \; \; \; \gyoung(;;,:;;,:::\vdots,:::;;,::::;;,:::::;,:::B) \; \; \; \; \gyoung(;,;;,:;;,:::\vdots,:::;;,::::;;,:::::;,:::C) \; \; \; \;  \gyoung(;;,:;;,:::\vdots,:::;;,::::;;,:::::;;,:::D)
\]
We denote by $\mc{A}_n, \mc{B}_n, \mc{C}_n, \mc{D}_n$ the corresponding diagram containing $n$ boxes, and have:
\begin{align*}
\mc{A}_n, \quad & n=2k+1, k \geq 1 \\
\mc{B}_n, \quad& n=2k+1, k \geq 1 \\
\mc{C}_n, \quad &  n=2l, l\geq 1 \\
\mc{D}_n, \quad &  n=2l, l \geq 1 
\end{align*}
We have in addition the generator $\mc{H}$ corresponding to the diagram with a single box, which is special. 

 When computing products/commutators in the Hall algebra of  $\Coh^T(\mathbb{A}^2)_{\m^2}$, we keep only those skew shapes annihilated by $\m^2$. We have for instance
\begin{align*}
 \mc{C}_n \bullet \mc{A}_m & = \mc{A}_{n+m} + \mc{C}_n \oplus \mc{A}_m \\
 \mc{A}_m \bullet \mc{C}_n &= \mc{C}_n \oplus \mc{A}_m
\end{align*}
which yields 
\[
[\mc{C}_n, \mc{A}_m ] = \mc{A}_{n+m}
\]
Similarly, we can compute the other commutators:
\begin{align*}
[\mc{C}_n, \mc{D}_m] &= 0 \\
[\mc{D}_n, \mc{A}_m] &= \mc{A}_{m+n} \\
[\mc{C}_n, \mc{B}_m] &= - \mc{B}_{m+n} \\
[ \mc{D}_n, \mc{B}_m] &= -\mc{B}_{m+n} \\
[\mc{A}_n, \mc{B}_m] &= - \mc{D}_{m+n} - \mc{C}_{m+n}  \\
[\mc{H}, \mc{A}_n] & = - [\mc{H}, \mc{B}_n] = \mc{D}_{n+1} + \mc{C}_{n+1} \\
[\mc{H}, \mc{C}_n] &= [\mc{H}, \mc{D}_n] = \mc{B}_{n+1} - \mc{A}_{n+1} \\
\end{align*}

The Lie algebra $\mathfrak{n}$ spanned by $\mc{H}, \mc{A}_{2j+1}, \mc{B}_{2k+1}, \mc{C}_{2r}, \mc{D}_{2s}, j \geq 1, k \geq 1, r \geq 2, s \geq 1$ is isomorphic to a Lie subalgebra $\mathfrak{k} \subset \mathfrak{gl}_2 [t]$:
\[
\begin{pmatrix} d(t) & a(t) \\ b(t)& c(t) \\ \end{pmatrix}
\]
where $a(t), b(t)$ are odd polynomials, with $\deg(a(t)) \geq 3, \deg(b(t)) \geq 3$, and $c(t), d(t)$ are even polynomials with $\deg(c(t)) \geq 2, \deg(d(t)) \geq 2$, via the isomorphism $\rho: \mathfrak{n} \rightarrow \mathfrak{k}$ given by:
\begin{align*}
\rho(\mc{H}) &= E_{21} \otimes t - E_{12} \otimes t \\
\rho(\mc{A}_n) &= - E_{12} \otimes t^n \\
\rho(\mc{B}_n) &= E_{21} \otimes t^n \\
\rho(\mc{C}_n) &= - E_{22} \otimes t^n \\
\rho(\mc{D}_n) &= E_{11} \otimes t^n
\end{align*}
where $E_{ij}$ denotes the $2 \times 2$ matrix with a $1$ in row $i$, column $j$, and $0$'s everywhere else. 

\subsection{Truncations of point sheaves on $\mathbb{A}^2$ by height}

We define the \emph{height} (or number of parts) of  a connected two-dimensional skew shape $\S$ as the number of horizontal segments making up the shape, and denote it by $ht(\S)$. For example, for 
\begin{center}
\[
\Ylinethick{1pt}
\S = \gyoung(;;,:;;;) \; \;
\T= \gyoung(;;,;;;;,:;;;;)
\]
\end{center}
we have $ht(\S) = 2$ and $ht(\T) = 3$ respectively. The height of a general skew shape is defined as the maximum of the heights of the connected components, and hence $ht(\S \oplus \T) = 3$. We also set $ht(M_{\S}) := ht(\S)$. Let $\Coh^T(\mathbb{A}^2)_{\leq m}$ denote the subcategory of $\Coh^T(\mathbb{A}^2)$ generated by finite modules of height $\leq m$. The following is straightforward:

\begin{prop} \label{height_hall_algebra}
Let $m$ be a positive integer and set $$I_m = \textrm{span} \{ \delta_{[M_{\S}]} \mid ht(\S) > m \} \subset H_0^T[2].$$ Then
\begin{enumerate}
\item $\Coh^T(\mathbb{A}^2)_{\leq m}$ is a proto-abelian subcategory of $\Coh^T(\mathbb{A}^2)$.
\item $I_m$ is a Hopf ideal in $H_0^T [2]$, and if $m < m'$, then $I_{m'} \subset I_m$. 
\item The Hall algebra $\H_{\Coh^T(\mathbb{A}^2)_{\leq m}}$ may be identified with $H_0^T [2] / I_m$
\end{enumerate}
\end{prop}

By the last part of Proposition \ref{height_hall_algebra} we may identify $\H_{\Coh(\mathbb{A}^2)^T_{\leq m}} \simeq U(\mathfrak{sk}^{\leq m}_2)$, where the Lie algebra $\mathfrak{sk}^{\leq m}_2$ is spanned by connected two-dimensional skew shapes of height $\leq m$, and in computing products, shapes of height $> m$ are discarded. We have surjective Lie algebra homomorphisms $ \mathfrak{sk}^{\leq m'}_2  \twoheadrightarrow \mathfrak{sk}^{\leq m}_2$ for $m \leq m'$ obtained by discarding the shapes of height $> m$. 

We examine here in detail the case of the Lie algebra $\mathfrak{sk}^{\leq m}_2$, for $m=1, 2$. The skew shapes of height $1$ are simply horizontal strips of length n:
\[
h_n =  \gyoung(;;:\cdots;;;) 
\]
In $U(\mathfrak{sk}^{\leq 1}_2)$, we have $$ h_n \bullet h_m = h_{m+n} + h_m \oplus h_n.$$ $\H_{\Coh(\mathbb{A}^2)^T_{\leq m}} \simeq U(\mathfrak{sk}^{\leq 1}_2) \simeq \Lambda$ by mapping $h_n$ to the $n$th power sum, where $\Lambda$ is the ring of symmetric functions viewed as a Hopf algebra. 
 
The shapes of height $2$ are of the form
\[
Y_{i,r,j} =  \gyoung(;;:\cdots;;;::j,::i:;;;;:\cdots;;) 
\]
where the numbers $i, r, j$ denote respectively:
\begin{itemize}
\item $i$ - the defect in the lower left hand corner (i.e. the number of boxes "missing" in the lower left hand corner),
\item $j$ - the defect in the upper right hand corner,
\item $r$ - the total number of boxes in the diagram. Note that $r=i+j + 2p$, where $p \geq 1$ is the number of vertical segments of length 2.  
\end{itemize}

We have the following brackets in $\mathfrak{sk}^{\leq 2}_2$:
\begin{align*}
[Y_{i, r, j}, Y_{i', s, j'} ] & = \delta_{j, i'} Y_{i, r+s, j'} - \delta_{j', i} Y_{i', r+s, j} \\
[h_n, h_m] &= \sum^{\min(m,n)}_{k=1} Y_{m-k,m+n,n-k} -  \sum^{\min(m,n)}_{k=1} Y_{n-k,m+n,m-k}\\
[h_n, Y_{i,r,j}] &= Y_{i+n, r+n, j} - Y_{i,r+n,j+n} + \theta(i-n) Y_{i-n,r+n,j} - \theta(j-n) Y_{i,r+n, j-n}
\end{align*}
where $\theta(x) = 1$ if $x \geq 0$ and $0 $ otherwise.

\begin{myeg} We have for instance in $\mathfrak{sk}^{\leq 2}_2$:

\[
\left[ \gyoung(;;;;,::;;;), \gyoung(;;,:;;;) \right] = \gyoung(;;;;;;,::;;;;;;) - \gyoung(;;;;;;,:;;;;;;)
\]
\[
\left[ \gyoung(;;), \gyoung(;;;;,::;;;) \right] = \gyoung(;;;;;;,::::;;;) + \gyoung(;;;;,;;;;;) - \gyoung(;;;;,::;;;;;)
\]
\[
\left[ \gyoung(;;), \gyoung(;;;) \right] = \gyoung(;;;,::;;) + \gyoung(;;;,:;;) - \gyoung(;;,:;;;) - \gyoung(;;,;;;)
\]

\end{myeg}

Let $\mathfrak{gl}_{\infty}$ denote the Lie algebra of infinite matrices with rows and columns indexed by $\mathbb{Z}_{\geq 0 }$, and $0$'s except on finitely many diagonals. I.e. for every $M \in \mathfrak{gl}_{\infty} $, there exists $n(M) \in \mathbb{N}$ such that $M_{i,j} = 0$ whenever $|i-j| \geq n(M)$. There is an embedding $$ \phi: \mathfrak{sk}^{\leq 2}_2 \hookrightarrow \mathfrak{gl}_{\infty} [t] $$ defined as follows:
\begin{align*}
\phi(Y_{i,r,j}) &= E_{i,j} \otimes t^r \\
\phi(h_n) & = \sum^{\infty}_{k=0} ( E_{k+n,k} \otimes t^n + E_{k,k+n} \otimes t^n )
\end{align*}

We note the following:

\begin{itemize}
\item $J_2 = span \{ Y_{i,r,j} \}$ is a Lie ideal in $\mathfrak{sk}^{\leq 2}_2$ and generates a Hopf ideal $\wt{J}_2$ in $U(\mathfrak{sk}^{\leq m}_2)$ such that 
\[
U(\mathfrak{sk}^{\leq m}_2) / \wt{J}_2 \simeq \Lambda
\]
\item $span \{ Y_{i, r, i} \}, i \geq 0$ is an infinite-dimensional maximal abelian Lie subalgebra of $\mathfrak{sk}^{\leq 2}_2 $
\end{itemize}

\subsection{The case of $\mathbb{P}^1$}

Referring back to Section \ref{Pone} and Proposition \ref{Pone_alpha}, we see that passing from $\Coh^{\alpha}(\Pone)$ to $\Coh^T(\Pone)$ has the effect of discarding the cyclotomic sheaves $\mc{C}_i$ while keeping $\T_{0,r}, \T_{\infty, s}$ and $\mc{O}(n)$. This has the effect of eliminating the factor $\mathfrak{k}$ from the Hall Lie algebra $\n^{\alpha}_{\Pone}$:

\begin{prop} \label{Pone_T}
$H^{T}_{\Pone}$ is isomorphic as a Hopf algebra to $U(\mathfrak{a})$, where $\mathfrak{a}$ is as in Proposition \ref{Pone_alpha}. 
\end{prop}

\subsection{The case of $\mathbb{P}^2$}

In this section we take up the example of $\H^T_{\mathbb{P}^2}$. We view $\mathbb{P}^2$ as equipped with the monoid scheme structure from Example \ref{P2}, and recall that it is covered by three affine open subsets isomorphic to $\mathbb{A}^2$:
\begin{enumerate}
	\item 
	$X_{\sigma_0}=\spec ( S_{\sigma_0} = \langle x_2, x_1 \rangle) \simeq \mathbb{A}^2$, 
	\item 
	$X_{\sigma_1}=\spec ( S_{\sigma_1} = \langle x_1^{-1}, x_1^{-1}x_2 \rangle) \simeq \mathbb{A}^2$, 
	\item 
$X_{\sigma_2}=\spec ( S_{\sigma_2} = \langle x_1x_2^{-1}, x_2^{-1}\rangle) \simeq \mathbb{A}^2$. 
\end{enumerate}

The only closed, reduced, and irreducible monoid subschemes of $\mathbb{P}^2$ are the three closed points $p_0, p_1, p_2$ corresponding to the origins of the each $\mathbb{A}^2$, and three $\mathbb{P}^1$'s $l_{01}, l_{02}, l_{12}$ joining these. We may visualize this as follows:
\[
\begin{tikzpicture}
\draw (0,0) node[anchor=north]{$p_0$}
  -- (4,0) node[anchor=north]{$p_1$}
  -- (2,3) node[anchor=south]{$p_2$}
  -- cycle;
 \draw (0.6,1.7) node[anchor=south]{$l_{02}$};
  \draw (3.4,1.7) node[anchor=south]{$l_{12}$};
 \draw (2,-0.5) node[anchor=south]{$l_{01}$};

\end{tikzpicture}
\]

\subsubsection{Indecomposable $T$-sheaves on $\mathbb{P}^2$}

By Theorem \ref{finitary_T} $\H^T_{\mathbb{P}^2} \simeq U(\mathfrak{n}^T_{\mathbb{P}^2})$, where $\mathfrak{n}^T_{\mathbb{P}^2}$ (which is denoted by $\mathfrak{n}$ from now on) has as basis the indecomposable $T$-sheaves on $\mathbb{P}^2$. We begin by enumerating these. In order to do so, it is useful to classify connected two-dimensional (possibly infinite) skew shapes $\S$ by their asymptotic behavior along the axes. This asymptotic behavior is characterized by a pair of extended integers $(\S_x, \S_y)$, $0 \leq \S_x, \S_y \leq \infty$ (note that $\S_x, \S_y$ may take the value $\infty$), where $\S_x$ (resp. $\S_y$) denotes the number of boxes in the intersection of $\S$ with an infinite vertical (resp. horizontal) strip as these move out to $\infty$. For example, the infinite shape $\S$:

\begin{center}
\Ylinethick{1pt}
\gyoung(:\vdots:\vdots,;;,;;,;;,;;;;;:\cdots,:;;;;:\cdots,::;;;:\cdots)
\end{center}
\Ylinethick{1pt}
\vspace{0.1cm}
has asymptotics $\S_x=3, \S_y = 2$. In what follows, we will say skew shapes rather than generalized skew shapes, and when we wish to emphasize the finiteness, we will say finite skew shapes. The following is straightforward:

\begin{prop} \label{ss_classification}
Let $\S$ be a connected skew shape in $A=\fun \langle x, y \rangle $. Then exactly one of the following is true:
\begin{itemize}
\item $(\S_x, \S_y) = (0,0)$. In this case $S$ is a finite skew shape, and $\wt{M}_{\S}$ is a torsion sheaf supported at the origin in $\mathbb{A}^2$.  
\item $(\S_x, \S_y) = (k, 0)$ for some $1 \leq k < \infty$. In this case $\S$ is infinite, and $\wt{M}_{\S}$ is a torsion sheaf supported along the $x$-axis in $\mathbb{A}^2$. 
\item $(\S_x, \S_y) = (0, k)$ for some $1 \leq k < \infty$. In this case $\S$ is infinite, and $\wt{M}_{\S}$ is a torsion sheaf supported along the $y$-axis in $\mathbb{A}^2$. 
\item $(\S_x, \S_y) = (k, l) $ {for some $1 \leq k,l < \infty$}. In this case $\S$ is infinite, and $\wt{M}_{\S}$ is {a torsion sheaf} supported along the union of the axes in $\mathbb{A}^2$. 
\item $(\S_x, \S_y) = (\infty, \infty)$. In this case $\S$ is infinite, $\wt{M}_{\S}$ is torsion-free, and $supp(\wt{M}_{\S}) = \mathbb{A}^2$.
\end{itemize}
\end{prop}

We will refer to the pair $(\S_x, \S_y)$ as the \emph{type} of a skew shape or corresponding module. If $A' \simeq \fun \langle x, y \rangle$ is another free commutative monoid on two generators, and $\S \in A'$ a skew shape, the type will depend on the choice of generators (there are two choices corresponding to ordering). 

We will view a coherent sheaf on $\mathbb{P}^2$ in terms of gluing data on the three patches $X_{\sigma_i}, i=0,1,2$. In other words, a $T$-sheaf $\F$ on $\mathbb{P}^2$ is given by the data $(M_{i}, \phi_{ij}), i, j \in \{ 0,1,2 \}$, where each $M_i$ is a direct sum of generalized skew shapes in $S_{\sigma_i}$, and 
\begin{equation} \label{P2_gluing}
\phi_{ij}: \wt{M}_{j} \vert_{X_{\sigma_i} \cap X_{\sigma_j} } \rightarrow \wt{M}_{i} \vert_{X_{\sigma_i} \cap X_{\sigma_j} } 
\end{equation}
 subject to the the cocycle condition. A subsheaf $ \F' \subset \F $ thus corresponds to submodules $N_i \subset M_i$ compatible with the $\phi_{ij}$. 

We note that each of the intersections $X_{\sigma_i} \cap X_{\sigma_j}$ corresponds to a distinguished open affine subset of each copy of $\mathbb{A}^2$. The following shows that indecomposables on $\mathbb{A}^2$ are indecomposable or $0$ when restricted to these. 

\begin{prop} \label{stays_indecomposable}
Let $\T$ be a connected skew shape in  $A=\fun \langle x, y \rangle $, and $S \subset A$ a multiplicative subset. Then $S^{-1} M_{\T}$ is either 0 or indecomposable. 
\end{prop}

\begin{proof}
Let $\wt{S} \subset A$ be the set of divisors of elements of $S$ - i.e. $\wt{S} := \{ a \in A \vert \exists b \in A \textrm{ such that } ab \in S \}$. Then  $S^{-1}M_{\T} = \wt{S}^{-1} M_{\T}$. $\mathfrak{p} = A \backslash \wt{S}$ is an ideal, which is prime since $S$ and therefore $\wt{S}$ are multiplicatively closed. It follows that $\mathfrak{p}$ is one of $( 0 ), ( x ) , ( y ), ( x,y )$. One can check directly using the classification in Proposition \ref{ss_classification} that the in each case $\wt{S}^{-1} M_{\T}$ is indecomposable or $0$.
\end{proof}

We can now proceed to classify the indecomposable $T$-sheaves on $\mathbb{P}^2$.
We introduce three classes of indecomposable sheaves, having zero, one, and two-dimensional support respectively. 

\begin{enumerate}
\item Let $\S$ be a connected skew shape of type $(0,0)$ in $\mathbb{F}_1\langle x_1, x_2 \rangle \backslash \{0\}$. Let $M_0 = M_{\S}, M_1 = 0, M_2=0$, and $\phi_{ij} = 0$ for all $i \neq j$. We denote this sheaf $\mathfrak{T}_{\S, 0}$. This is a torsion sheaf supported at $p_0$. $\mathfrak{T}_{\S, i}, i=1,2$ may be defined similarly, yielding torsion sheaves supported at $p_1$ (resp. $p_2$). $\mathfrak{T}_{\S, i}$ is clearly indecomposable.

\item We introduce two types of $T$-sheaves supported along the triangle formed by the three lines $l_{01}, l_{02}, l_{12}$.

We begin by observing that if $\mc{S}_0 \subset S_{\sigma_0} = \mathbb{F}_1\langle x_2, x_1 \rangle \backslash \{0\}$ and $\mc{S}_1 \subset  S_{\sigma_1} =\mathbb{F}_1 \langle x_1^{-1}, x_1^{-1}x_2 \rangle \backslash \{0\}$ are two connected skew shapes of types $(a,b)$ and $(c,d)$ respectively, and $a,b,c,d < \infty$, then an isomorphism
\[
\phi_{01}: \wt{M}_{\S_0} \vert_{X_{\sigma_0} \cap X_{\sigma_1} } \rightarrow \wt{M}_{\S_1} \vert_{X_{\sigma_0} \cap X_{\sigma_1} } 
\]
exists if and only if the matching condition $b=c$ holds. If $b=c > 0$, then $\phi_{01}$ is given by an integer which is the degree of the restriction of the glued sheaf to the line $l_{01}$. If $b=c=0$, then $\phi_{01}$ is necessarily zero. The same observation obviously applies to any pair of patches $i, j \in \{0,1,2\}$. We note also if $\mc{S}_i \subset S_{\sigma_i}$ has finite asymptotics $(a,b)$, then $\wt{M}_{\S_i}$ restricted to the triple intersection $X_{\sigma_0} \cap X_{\sigma_1} \cap X_{\sigma_2}$ vanishes, which means that the cocycle condition for sheaves built out of these modules is trivial - we need only specify $\phi_{01}, \phi_{02}, \phi_{12}$ without any restrictions. 

For $i,j \in \{0,1,2 \}$ viewed as residues $ \on{mod} 3$, consider a finite sequence of connected skew shapes {of type $(a,b)$, where $a,b < \infty$:}
\[
\mc{S}_{i, 1}, \S_{i+1,2}, \cdots, \mc{S}_{j-1,q-1}, \mc{S}_{j,q}
\] 
where $q \geq 1$, and integers $n_{r}, 1 \leq r \leq q-1$, such that
\begin{itemize}
\item For each $1 \leq e \leq q$  $\mc{S}_{k, e} \subset S_{\sigma_k}$, where $k$ is viewed as residues mod $3$.
\item $\mc{S}_{k,r}$ and $\mc{S}_{k+1,r+1}$ have matching non-zero type along $l_{k,k+1}$. 
\item $\S_{i,1}$ has type $(0,a), a > 0$ and $\S_{j,q}$ has type $(b,0), b > 0$. 
\end{itemize}
Gluing $\wt{M}_{\S_{k,r}}$ to $\wt{M}_{\S_{k+1,r+1}}$ using the $\phi_{k,k+1}$ specified by $n_r$, we obtain a sheaf denoted by $\mathfrak{H}[i,j,q,n_r, \mc{S}_{k,r}]$, which we call a \emph{helix}. As the name suggests, it can be thought of as a staircase built out of skew shapes which begins at $i \in \{0,1,2\}$ and winds {counter-clockwise}  around the triangle of $\mathbb{P}^1$'s, ending at $j$. All helices can be made to wind counter-clockwise by reading the list backwards if necessary. We note that the requirements in the third bullet point above are made necessary if we are to assign $i$ as the "starting" point of the helix and $j$ as the "end" point via the matching condition - i.e. if $\S_{i,1}$ has type $(a',a)$ with $a' > 0$, then it must be glued to a skew shape of type $\S_{i-1, 0}$ of type $(a'',a')$ on $X_{\sigma_{i-1}} \cap X_{\sigma_{i}}$, which in turn means that that $i$ was not the starting point.

Suppose instead that in the above data we have $j=i-1 \on{mod} 3$, and that $\S_{i,1}, \S_{j,q}$ have types $(a,b)$, $(a',b')$ with all $a,b,a',b' > 0$ with matching type along $l_{i-1,i}$, and that we are additionally given an integer $n_q$. Note that in this case $q$ is a multiple of $3$. Performing one additional gluing of $\wt{M}_{\S_{j,q}}$ to $\wt{M}_{\S_{i,1}}$ using $n_q$, we obtain a sheaf denoted by $\mathfrak{C}[i,j,q,n_r, \mc{S}_{k,r}]$, which we call a \emph{cycle}. It winds $q/3$ times around the triangle analogous to the covering $z \rightarrow z^{q/3}$ of the punctured plane. 

The sheaves $\mathfrak{H}[i,j,q,n_r, \mc{S}_{k,r}]$, $\mathfrak{C}[i,j,q,n_r, \mc{S}_{k,r}]$ are easily seen to be indecomposable under the above conditions. 

\item Let $\S_0 \subset \mathbb{F}_1\langle x_2, x_1 \rangle \backslash \{0\}$, $\S_1 \subset \mathbb{F}_1\langle x_1^{-1}, x_1^{-1}x_2 \rangle \backslash \{0\} $, $\S_2 \subset \mathbb{F}_1\langle x_1x_2^{-1}, x_2^{-1}\rangle \backslash \{0\}$ be three connected skew shapes of types $(\infty, \infty)$. To describe the gluing data, we first minimally complete each shape $\S_i$ to $\ol{\mc{S}_i}$ corresponding to a free module with generator $g_i$ by "filling in" the corner. For the shape $\T$ in Example \ref{tf_example}, this looks as follows:

\begin{center}
\Ylinethick{1pt}
\gyoung(:\vdots:\vdots:\vdots:\vdots:\vdots:\vdots:\iddots,;;;;;;:\cdots,;;;;;;:\cdots,:a;;;;;:\cdots,:a;;;;;:\cdots,:g:a:a:a;;:\cdots)
\end{center}
where the boxes added to $\ol{\T}$ are labeled $a$, and $g$ for the generator. We thus have $$\wt{M}_{\ol{S_0}}\vert_{X_{\sigma_0} \cap X_{\sigma_1} } \simeq \wt{M}_{\ol{S_1}}\vert_{X_{\sigma_0} \cap X_{\sigma_1} }  \simeq \wt{\mathbb{F}_1\langle x_1, x^{-1}_1, x_2 \rangle } $$ and the isomorphism $\phi_{01}$ is specified by an integer $m_{01}$ such that $g_0 = x^{-m_{01}}_1 g_1$. Similarly, the isomorphisms $\phi_{12}$ and $\phi_{02}$ are specified by integers $m_{12}, m_{02}$ such that $g_1 = (x^{-1}_1 x_2)^{-m_{12}} g_2$ and $g_0 = x^{-m_{02}}_2 g_2$. The cocycle condition now forces $m_{01}=m_{02}=m_{12} = m$. We denote the resulting sheaf by $\mathfrak{TF}[\S_i, m]$. It is easily seen to be a torsion-free indecomposable sheaf. By construction, it embeds minimally into the line bundle $\mc{O}(m)$, and there is an exact sequence
\[
\ses{\mathfrak{TF}[\S_i, m]}{\mc{O}(m)}{\mathfrak{Q}}
\]
where $$\mathfrak{Q} = \mathfrak{T}_{\mc{U}_0, 0} \oplus \mathfrak{T}_{\mc{U}_1, 1} \oplus \mathfrak{T}_{\mc{U}_2, 2}  $$ where the $\mathfrak{T}_{\mc{U}_i, i} $ are point sheaves as in $(1)$, and $\mc{U}_i = \ol{\S_i} \backslash \S_i$ are finite shapes. We can visualize the sheaves $\mathfrak{TF}[\S_i, m]$ as follows:

\[
\begin{tikzpicture}
\draw (0,0) node[anchor=north]{$\S_0$}
  -- (4,0) node[anchor=north]{$\S_1$}
  -- (2,3) node[anchor=south]{$\S_2$}
  -- cycle;
 \draw (2,1.5) node[anchor=north]{$m$};
\end{tikzpicture}
\]

\end{enumerate}

The following proposition shows that the above list is exhaustive. 

\begin{prop} \label{P2_classification}
Let $\F$ be an indecomposable sheaf in $\Coh^T (\mathbb{P}^2)$. Then $\F$ is isomorphic to one of $\mathfrak{T}_{\S, i}$,  $\mathfrak{H}[i,j,q,n_r, \mc{S}_{k,r}]$, $\mathfrak{C}[i,j,q,n_r, \mc{S}_{k,r}]$, $\mathfrak{TF}[\S_i, m]$.
\end{prop}

\begin{proof}

Suppose that $\mc{F}$ is an indecomposable coherent $T$-sheaf described in terms of gluing data by $(M_i, \phi_{ij}), i, j \in \{ 0, 1, 2 \}$, with $M_i = \oplus^{v_i}_{r=1} M_{\S_{i,r}}$, for connected skew shapes ${\S_{i, r} \subset S_{\sigma_i}}$.

\begin{itemize}
\item Suppose that one of the $M_i$'s, say $M_0$, has an indecomposable summand of type $(0,0)$, denoted $M'_0$. Write $M_0 = M'_0 \oplus M''_0$. Since $\wt{M}'_0$ restricted to $X_{\sigma_0} \cap ( X_{\sigma_1} \cup X_{\sigma_2}) $ vanishes, the proof of Proposition \ref{splitting_sheaves} implies that $M'_0$ generates a summand of the form $\mathfrak{T}_{\S, 0}$. Since $\F$ is indecomposable, we must have $\F = \mathfrak{T}_{\S, 0}$.
\item Suppose that one of $M_i$'s, say $M_0$ has an indecomposable summand $N_0$ of type $(\infty, \infty)$, and write $M_0 = N_0 \oplus K_0$. $\wt{N}_0 \vert_{X_{\sigma_0} \cap X_{\sigma_i} }$ is isomorphic to the structure sheaf, and from the above classification and Proposition \ref{splitting_isomorphisms} it follows that $M_1=N_1 \oplus K_1$, $M_2 = N_2 \oplus K_2$ where $N_i, i=1,2$ are also of type $(\infty, \infty)$, and $\phi_{02}: \wt{N}_2 \simeq \wt{N}_0$, $\phi_{01}: \wt{N}_1 \simeq \wt{N}_0$ are isomorphisms over $X_{\sigma_0} \cap X_{\sigma_i} $. The cocycle condition for $\{ \phi_{ij} \}$ the forces $\phi_{12}: \wt{N}_2 \rightarrow \wt{N}_1$ on $X_{\sigma_1} \cap X_{\sigma_2}$ to be an isomorphism. It follows that $\F$ has a summand of the form $\mathfrak{TF}[\S_i, m]$, and since $\F$ is indecomposable, must be isomorphic to it. 
\item If $\F$ is not isomorphic to $ \mathfrak{T}_{\S, i}$ or $\mathfrak{TF}[\S_i, m]$, then each $M_i$ must be a sum of modules having finite type $(a,a')$ with at least one of $a, a'$ nonzero. For each $i \in \{ 0,1,2 \}$, write $$M_i =  M_{\S_{i,1}} \oplus M_{\S_{i,2}}  \oplus \cdots \oplus M_{\S_{i,m_i}} $$ with the $\S_{i,k}$ connected. 
We note that if $\S_{i,k}$ has nonzero type along $l_{i,i+1}$, then $\phi_{i,i+1}$ yields an isomorphism 
\begin{equation} \label{edge_iso}
\phi_{i,i+1}: \wt{M}_{\S_{i,k}} \vert_{X_{\sigma_i} \cap X_{\sigma_{i+1}} } \rightarrow \wt{M}_{\S_{i+1,k'}} \vert_{X_{\sigma_i} \cap X_{\sigma_{i+1}} } 
\end{equation}
for a unique $M_{\S_{i+1,k'}}$

Consider the colored directed graph $\Gamma_{\F}$ whose vertices are the $M_{\S_{i,k}}$'s, and there is a directed edge from $M_{\S_{i,k}}$ to $M_{\S_{i+1,k'}}$ if $\S_{i,k}$ has nonzero type along $l_{i,i+1}$ and  \eqref{edge_iso} holds, colored by the degree of the gluing map along $l_{i,i+1}$.  $\Gamma_{\F}$ is easily seen to have the property that each vertex has at most one incoming and at most one outgoing edge. It is also clear that $\F$ is indecomposable if and only if $\Gamma_{\F}$ is connected. The only connected directed graphs having this property are either ladders or pure cycles. If $\Gamma_{\F}$ is a ladder, then $\F$ is a helix (i.e. of type $\mathfrak{H}$), and if a pure cycle, a cycle (i.e. of type $\mathfrak{C}$). 

\end{itemize}

\end{proof}

\begin{rmk}
It follows from the above classification that every indecomposable torsion-free $T$-sheaf is of rank one. In particular, every locally free sheaf is a direct sum of line bundles, as shown in \cite{GvB, Pir}
\end{rmk}

By Proposition \ref{P2_classification}, the Hall Lie algebra $\n$ here is spanned by delta-functions supported on isomorphism classes of the sheaves  $\mathfrak{T}_{\S, i}$,  $\mathfrak{H}[i,j,q,n_r, \mc{S}_{k,r}]$, $\mathfrak{C}[i,j,q,n_r, \mc{S}_{k,r}]$, $\mathfrak{TF}[\S_i, m]$ above. It appears difficult to describe in explicit terms. We restrict ourselves to a few remarks, intending to study its structure more closely in a future paper. 

\begin{itemize}
\item $\n$ contains three commuting copies $\mathfrak{t}_0, \mathfrak{t}_1, \mathfrak{t}_2$ of the Lie algebra $\mathfrak{sk}_2$ {from Section \ref{skew_shapes}}, with $\mathfrak{t}_i$ spanned by the sheaves $\mathfrak{T}_{\S, i}$. Each $\mathfrak{t}_i$ acts on the sheaves with higher-dimensional support via Hecke-type operators. 
\item Sheaves supported on each line $l_{i+1,i+2}$ opposite $p_i$ generate a Lie subalgebra $\mathfrak{l}_i$. These three subalgebras are no longer mutually commuting however. Each $\mathfrak{l}_i$ carries commuting actions of $\mathfrak{t}_{i+1}$ and $\mathfrak{t}_{i+2}$. 
\item The isomorphism classes of torsion-free sheaves $[\mathfrak{TF}[\S_i, m]]$  generate a commutative subalgebra of $H^T_{\mathbb{P}^2}$. To see this suppose that $\F, \F'$ are indecomposable and torsion-free, and
\begin{equation} \label{ses_tf}
\ses{\F}{\G}{\F'}
\end{equation}
is an exact sequence. By Lemma \ref{subobjects} and the above classification, $\F$ would have to embed as a sub-sheaf of some $\mathfrak{TF}[\S_i, m]$. The cokernel of this map, if non-zero, would be torsion. It follows that \ref{ses_tf} splits, showing that  $[\F], [\F']$ commute in $H^T_{\mathbb{P}^2}$.
\end{itemize}

\begin{myeg}
Consider the closed embedding $\iota_j : \mathbb{P}^1 \rightarrow \mathbb{P}^2$ whose image is the $\mathbb{P}^1$  $l_{j+1, j+2}$ opposite the vertex $p_j$, and let $\mc{O}_j (n)$ denote $\iota_{j,*} \mc{O}_{\mathbb{P}^1}(n)$. For instance, $\mc{O}_2 (n)$ may be described in terms of gluing data as $(M_{\S_0}, M_{\S_1}, M_{\S_2}, \phi_{ij})$, where $M_{\S_2} = 0$, $\S_0, \S_1$ are the infinite strips
\[
\S_0: \; \; \gyoung(g;;;;;;;;:\cdots)
\]
\[
\S_1: \; \; \gyoung(h;;;;;;;;:\cdots)
\]
where the generator $g$ of $M_{\S_0}$ is annihilated by $x_2$, the generator $h$ of $\M_{\S_1}$ is annihilated by by $x_1^{-1} x_2$, $\phi_{01}$ identifies $h$ with $x^n_1 g$, and $\phi_{20}$ and $\phi_{12}$ are zero. One can ask about the Lie subalgebra $\mathfrak{k}$ generated by $\{\mc{O}_i (n) \}, i \in \{ 0,1,2 \},  n \in \mathbb{Z}$ inside $\n$. 

Let $\mathbb{{TF}}$ denote the free commutative subalgebra of $H^T_{\mathbb{P}^2}$ generated by the isomorphism classes of torsion-free sheaves, and let $\mathbb{J} \subset \mathbb{{TF}}$ denote the ideal generated by the isomorphism classes of those torsion-free sheaves which are not locally free. Let $\Bun$ denote the set of isomorphism classes of locally free sheaves on $\mathbb{P}^2$, which are simply direct sums of line bundles. $\mathbb{TF}/\mathbb{J}$ is isomorphic to the polynomial ring on $x_n = [\mc{O}(n)], n \in \mathbb{Z}$, which in turn may be identified with the vector space of functions with finite support on $\Bun$, where the monomial $x_{n_1} x_{n_2} \cdots x_{n_r}$ is identified with the delta-function on $[\mc{O}(n_1) \oplus \mc{O}(n_2) \cdots \oplus \mc{O}(n_r)]$.  

One checks that for $i \in \{0,1,2 \}, n \in \mathbb{Z}$, $[\mc{O}_i (n), \mathbb{TF}] \subset \mathbb{TF}$ and $[\mc{O}_i (n), \mathbb{J}] \subset \mathbb{J}$. The adjoint action of $\mathfrak{k}$ therefore descends to $\mathbb{Q}[x_n] \simeq \mathbb{TF}/\mathbb{J}$ by derivations.  

There is for each $i, n$ a unique non-split extension
\[
\ses{\mc{O}(n-1)}{\mc{O}(n)}{\mc{O}_i (n)}
\]
For $m \neq n$, every extension 
\[
\ses{\mc{O}(n-1)}{\F}{\mc{O}_i (m)}
\]
with $\F$ locally free is split (though there are in general several torsion-free extensions). Similarly, every extension
\[
\ses{\mc{O}_i (m)}{\F}{{\mc{O}(n)}}
\]
is split.  This yields the identity
\begin{equation} \label{line_action}
[\mc{O}_i (n), \mc{O} (m) ] = \delta_{n-1,m} \mc{O}(n)
\end{equation}
$\on{mod} \mathbb{J}$, showing that under the above identifications, $\mc{O}_i(n)$ acts by $x_n \partial_{n-1}$. Letting $\wt{\mathfrak{gl}}_{\infty}$ denote the Lie algebra of infinite matrices spanned by $E_{i,j}, i,j \in \mathbb{Z}$, the previous analysis shows that there exists a surjective Lie algebra homomorphism
\[
\nu: \mathfrak k \twoheadrightarrow \wt{\mathfrak{gl}}^-_{\infty}
\]
where $\wt{\mathfrak{gl}}^-_{\infty}$ denotes the lower-triangular part, spanned by $E_{i,j}$ with $i > j$. The kernel of $\nu$ is large however, as it contains the ideal generated by $\mc{O}_i (n) - \mc{O}_j (n) $ for all $i,j \in \{ 0,1,2 \}, n \in \mathbb{Z}$. 
\end{myeg}

\address{\tiny Department of Mathematics, SUNY New Paltz, New Paltz, NY} \\
\indent \footnotesize{\email{junj@newpaltz.edu}}

\address{\tiny Department of Mathematics and Statistics, Boston University, 111 Cummington Mall, Boston} \\
\indent \footnotesize{\email{szczesny@math.bu.edu}}

\end{document}